\documentclass[10pt]{amsart}

\usepackage[latin2]{inputenc}
\usepackage{amsmath}
\usepackage{graphicx}
\usepackage{amssymb}
\usepackage{esint}
\usepackage[dvipsnames]{xcolor}
\usepackage{tikz}
\usepackage{xxcolor}
\usepackage{floatrow}
\usepackage{color}
\usepackage{amsthm}
\usepackage{epsfig}
\usepackage[english]{babel}
\usepackage{hyperref}

\newtheorem{theorem}{Theorem}
\newtheorem{assumption}[theorem]{Assumption}
\newtheorem{proposition}[theorem]{Proposition}
\newtheorem{lemma}[theorem]{Lemma}
\newtheorem{corollary}[theorem]{Corollary}
\newtheorem{definition}[theorem]{Definition}
\newtheorem{remark}[theorem]{Remark}

\newtheorem{notation}[theorem]{Notation}

\newtheorem*{theorem*}{Theorem}

\def\Xint#1{\mathchoice
	{\XXint\displaystyle\textstyle{#1}}%
	{\XXint\textstyle\scriptstyle{#1}}%
	{\XXint\scriptstyle\scriptscriptstyle{#1}}%
	{\XXint\scriptscriptstyle\scriptscriptstyle{#1}}%
	\!\int}
\def\XXint#1#2#3{{\setbox0=\hbox{$#1{#2#3}{\int}$ }
		\vcenter{\hbox{$#2#3$ }}\kern-.6\wd0}}

\def\dashint{\Xint-}

\allowdisplaybreaks[2]

\newcommand{\Id}{\operatorname{Id}} 
\newcommand{\supp}{\operatorname{supp}}

\newcommand{\dist}{\operatorname{dist}}

\newcommand{\RVE}{{\operatorname{RVE}}}
\newcommand{\SQS}{{\operatorname{sel-RVE}}}

\newcommand{\Var}{{\operatorname{Var}~}}
\newcommand{\Cov}{\operatorname{Cov}}

\newcommand{\ol}{\overline}
\newcommand{\wh}{\widehat}
\newcommand{\wt}{\widetilde}
\newcommand{\coleq}{\mathrel{\mathop:}=}

\newcommand{\Huloc}{H^1_{\operatorname{uloc}}(\mathbb{R}^3)}
\newcommand{\Luloc}{L^2_{\operatorname{uloc}}(\mathbb{R}^3)}

\definecolor{darkred}{rgb}{0.7, 0.0, 0.0}
\definecolor{darkblue}{rgb}{0.0, 0.0, 0.7}
\definecolor{Yellow}{rgb}{0.95,0.9,0.0} 
\definecolor{Red}{rgb}{0.8,0.1,0.1}
\definecolor{Green}{rgb}{0.1,0.65,0.2}
\definecolor{Blue}{rgb}{0.1,0.1,0.8}
\definecolor{Purple}{rgb}{0.7,0.1,0.7}
\definecolor{Grey}{rgb}{0.6,0.6,0.6}

\begin{document}

\title[Variance reduction in the Thomas--Fermi--von Weizs\"acker energy]{Variance reduction for effective energies\\of random lattices in the Thomas--\\Fermi--von~Weizs\"acker model}

\author{Julian Fischer}
\address{Institute of Science and Technology Austria (IST Austria),
Am Campus 1, 3400 Klosterneuburg, Austria, E-Mail:
julian.fischer@ist.ac.at}

\author{Michael Kniely}
\address{Institute of Science and Technology Austria (IST Austria),
Am Campus 1, 3400 Klosterneuburg, Austria, E-Mail:
michael.kniely@ist.ac.at}
	
\begin{abstract}
In the computation of the material properties of random alloys, the method of ``special quasirandom structures'' attempts to approximate the properties of the alloy on a finite volume with higher accuracy by replicating certain statistics of the random atomic lattice in the finite volume as accurately as possible.
In the present work, we provide a rigorous justification for a variant of this method in the framework of the Thomas--Fermi--von~Weizs\"acker (TFW) model.
Our approach is based on a recent analysis of a related variance reduction method in stochastic homogenization of linear elliptic PDEs and the locality properties of the TFW model. Concerning the latter, we extend an exponential locality result by Nazar and Ortner to include point charges, a result that may be of independent interest.
\end{abstract}
	
	
	
\maketitle

\section{Introduction}
\label{secintro}

In material science, direct simulations based on density functional theory \cite{HohenbergKohn,KohnSham,ParrYang} are currently limited to hundreds to thousands of atoms and therefore to material samples just about one order of magnitude larger than the atomic length scale (see e.\,g.\ \cite{DislocationCore}).
Multiscale approaches -- employed for example in the simulation of dislocations 
\cite{FagoHayesCarterOrtiz,LuTadmorKaxiras,
SuryanarayanaBhattacharyaOrtiz} -- rely on an extrapolation of the elastic properties of the material from such microscopic samples to larger scales, a concept also known in the context of continuum mechanics as ``method of representative volumes''.
While for materials with a periodic lattice the computational problem on the atomic scale may often be simplified to a problem on a single periodicity cell \cite{BlancLeBrisLions,FagoHayesCarterOrtiz,LuTadmorKaxiras,DislocationCore,
SuryanarayanaBhattacharyaOrtiz}, 
such a simplification is no longer possible for materials with random atomic lattices like random alloys (see Figure~\ref{FigureRandomLattice} for an illustration). As a consequence, for random alloys the atomic-scale samples must be chosen significantly larger, giving rise to a computationally costly problem.

For the computation of the effective properties of random alloys,
an approach called ``special quasirandom structures'' (SQS) has been proposed by Zunger et al.\ \cite{ZWFB90} to increase the accuracy of DFT computations without increasing computational effort.
The key idea of the method of special quasirandom structures is to construct a periodic configuration of atoms with finite but large periodicity cell (``superlattice'')
which reflects certain statistical properties of the random atomic lattice particularly well -- like the proportion of the atomic species, the proportion of nearest-neighbor contacts of the various atomic species, and so on (see Figure~\ref{FigureSQS} for an illustration). Further developments and applications of this method of ``special quasirandom structures'' may be found in \cite{vPDFN10,WFBZ90}. Related approaches have been employed in the context of homogenization in continuum mechanics \cite{BalzaniBrandsSchroederCarstensen2010,BalzaniSchroeder2009,Schroeder2011}.

Inspired by the method of special quasirandom structures, in the continuum mechanical context of homogenization of random materials a selection approach for representative volumes has been proposed by Le~Bris, Legoll, and Minvielle \cite{LLM16}: This selection approach proceeds by considering a large number of microscopic samples of the random material and selecting the sample that is ``most representative'' for the material as measured by certain statistical quantities, like for example the volume fraction in the case of a two-material composite. The effective material properties are then approximated by numerically evaluating the cell formula provided by homogenization theory on the selected sample. In the context of stochastic homogenization of linear elliptic PDEs $-\nabla \cdot (a_\varepsilon\nabla u)=f$, for the computation of the effective (homogenized) coefficient the selection approach has been shown to yield an increase in accuracy of up to one order of magnitude in a numerical example with ellipticity ratio $5$ \cite{LLM16}, while requiring negligible computational effort.
Recently, a rigorous mathematical analysis of the selection approach by Le~Bris, Legoll, and Minvielle in the context of homogenization of linear elliptic PDEs has been provided by the first author \cite{FischerVarianceReduction}.

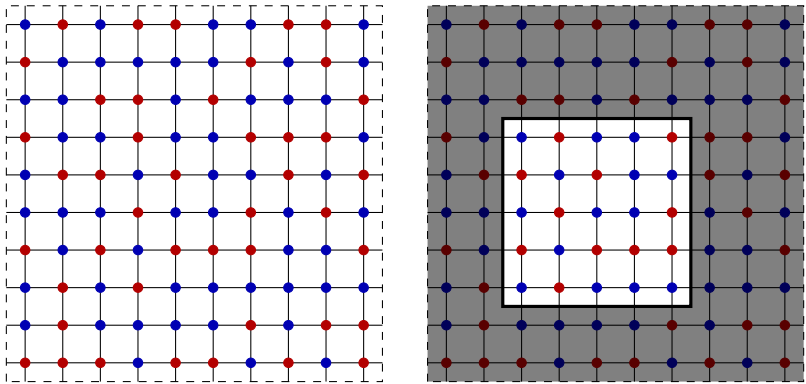
\begin{figure}
\begin{center}
\begin{tikzpicture}[scale=.5]
\draw[dashed] (-.5,-.5) rectangle (9.5,9.5);
\foreach \x in {0,...,9}
\draw (\x,-.5) -- (\x,9.5);
\foreach \y in {0,...,9}
\draw (-.5,\y) -- (9.5,\y);

\foreach \x/\y in {0/0, 0/3, 1/0, 1/1, 1/2, 2/0, 2/3, 3/2, 3/4, 4/0, 4/3, 
	5/0, 5/3, 6/1, 6/3, 6/4, 7/0, 8/1, 8/4, 9/0, 9/1, 9/3, 
	0/6, 0/8, 1/5, 1/9, 2/5, 2/7, 3/6, 3/7, 3/9, 4/5, 4/9, 
	5/7, 6/6, 6/8, 7/5, 7/6, 7/9, 8/6, 8/8, 8/9, 9/5, 9/7}
\draw[fill=darkred,color=darkred] (\x,\y) circle (.125);

\foreach \x/\y in {0/1, 0/2, 0/4, 1/3, 1/4, 2/1, 2/2, 2/4, 3/0, 3/1, 3/3, 4/1, 4/2, 4/4, 
	5/1, 5/2, 5/4, 6/0, 6/2, 7/1, 7/2, 7/3, 7/4, 8/0, 8/2, 8/3, 9/2, 9/4, 
	0/5, 0/7, 0/9, 1/6, 1/7, 1/8, 2/6, 2/8, 2/9, 3/5, 3/8, 4/6, 4/7, 4/8, 
	5/5, 5/6, 5/8, 5/9, 6/5, 6/7, 6/9, 7/7, 7/8, 8/5, 8/7, 9/6, 9/8, 9/9}
\draw[fill=darkblue,color=darkblue] (\x,\y) circle (.125);
\end{tikzpicture}
~~~
\begin{tikzpicture}[scale=.5]
\draw[dashed] (-.5,-.5) rectangle (9.5,9.5);
\foreach \x in {0,...,9}
\draw (\x,-.5) -- (\x,9.5);
\foreach \y in {0,...,9}
\draw (-.5,\y) -- (9.5,\y);
\draw[very thick] (1.5,1.5) -- (6.5,1.5) -- (6.5,6.5) -- (1.5,6.5) -- cycle;

\foreach \x/\y in {0/0, 0/3, 1/0, 1/1, 1/2, 2/0, 2/3, 3/2, 3/4, 4/0, 4/3, 
	5/0, 5/3, 6/1, 6/3, 6/4, 7/0, 8/1, 8/4, 9/0, 9/1, 9/3, 
	0/6, 0/8, 1/5, 1/9, 2/5, 2/7, 3/6, 3/7, 3/9, 4/5, 4/9, 
	5/7, 6/6, 6/8, 7/5, 7/6, 7/9, 8/6, 8/8, 8/9, 9/5, 9/7}
\draw[fill=darkred,color=darkred] (\x,\y) circle (.125);

\foreach \x/\y in {0/1, 0/2, 0/4, 1/3, 1/4, 2/1, 2/2, 2/4, 3/0, 3/1, 3/3, 4/1, 4/2, 4/4, 
	5/1, 5/2, 5/4, 6/0, 6/2, 7/1, 7/2, 7/3, 7/4, 8/0, 8/2, 8/3, 9/2, 9/4, 
	0/5, 0/7, 0/9, 1/6, 1/7, 1/8, 2/6, 2/8, 2/9, 3/5, 3/8, 4/6, 4/7, 4/8, 
	5/5, 5/6, 5/8, 5/9, 6/5, 6/7, 6/9, 7/7, 7/8, 8/5, 8/7, 9/6, 9/8, 9/9}
\draw[fill=darkblue,color=darkblue] (\x,\y) circle (.125);

\fill[opacity=.5] (1.5,-.5) -- (9.5,-.5) -- (9.5,9.5) -- (-.5,9.5) -- (-.5,-.5) -- (1.5,-.5) -- (1.5,6.5) -- (6.5,6.5) -- (6.5,1.5) -- (1.5,1.5) -- (1.5,-.5);
\end{tikzpicture}
\caption{A simple example of a random atomic lattice, the different atomic species being indicated by the colors red and blue (left).
An illustration of the method of representative volumes (right): For ab initio computations of material properties, a sample of microscopic extent must be chosen.
\label{FigureRandomLattice}}
\end{center}
\end{figure}

The main goal of the present paper is to show that the selection approach of Le~Bris, Legoll, and Minvielle \cite{LLM16}
-- which is conceptually closely related to the method of special quasirandom structures of Zunger et al.\ \cite{ZWFB90} -- also allows for an increase of accuracy in the computation of the effective elastic properties of random atomic lattices in the context of orbital-free density functional theory (orbital-free DFT). More precisely, we neglect exchange-correlation energy and consider the approximation of effective energies of random atomic lattices in the framework of the Thomas--Fermi--von Weizs\"acker (TFW) model. In the TFW model, for a given nuclear charge distribution $m$ the associated electronic density $\rho$ of the ground state is determined by minimizing the TFW energy
\begin{align*}
\int C_W |\nabla \sqrt{\rho}|^2 + C_{TF} \rho^{5/3} + \frac{1}{2} (m-\rho) \phi \,dx
\end{align*}
with the electric potential $\phi$ being subject to the Poisson equation
\begin{align*}
-\Delta \phi = 4\pi (m-\rho).
\end{align*}
By rescaling, we may henceforth assume that $C_W=1$ and $C_{TF}=1$.
Recall that it is convenient to reformulate the TFW model in terms of the square root of the electronic density $u:=\sqrt{\rho}$. With this notation, the Euler-Lagrange equation for the TFW model reads
\begin{subequations}
\label{eqtfw}
\begin{align}
&-\Delta u + \frac53 u^\frac73 - \phi u = 0,
\\
&-\Delta \phi = 4\pi (m - u^2).
\end{align}
\end{subequations}
In orbital-free DFT, further contributions accounting for exchange and correlation energy are typically added to the TFW energy (and, corrrespondingly, to the Euler-Lagrange equation). In the present work, we shall neglect those terms.
We will also assume that the positions of the nuclei are given a priori.
While in a more realistic model the positions of the nuclei would be determined by energetic relaxation, the question of crystallization in variational models of interacting atoms is a challenging topic on its own, with positive answers currently restricted to rather elementary (mostly non-quantum mechanical) models; see e.\,g.\ the review \cite{BlancLewin}. For this reason, we restrict ourselves to the aforementioned setting of fixed nuclei positions.
For an overview of the mathematical theory of the TFW model, we refer to \cite{CLBL98} and the references therein.

Recall that in the framework of hyperelasticity the deformation of an elastic body is determined by minimization of the total (elastic and potential) energy.
In the context of an atomic lattice, the elastic energy is given as the overall energy of electrons and nuclei.
In a multiscale approximation, the macroscopic deformation of the elastic body is approximated on the atomic length scale by affine deformations. In many cases, for a macroscopically affine deformation the state of minimal energy of the atomic lattice is given by an approximately affinely deformed atomic lattice (a principle known as the Cauchy-Born rule, see e.\,g.\ \cite{ContiDolzmannKirchheimMueller,FrieseckeTheil}).
The associated effective (homogenized) elastic energy density is then given by the thermodynamic limit (i.\,e.\ the ``infinite-volume average'') of the energy of the affinely deformed atomic lattice.
In other words, in the context of the TFW approximation the effective elastic energy density is given as the thermodynamic limit of the TFW energy, i.\,e.\ -- up to subtracting the self-energy of point charges -- by the quantity
\begin{align}
\label{DefinitionEeff}
E_{\infty} := \lim_{R\rightarrow \infty} \dashint_{[-R,R]^d} |\nabla \sqrt{\rho}|^2 + \rho^{5/3} + \frac{1}{2} (m-\rho) \phi \,dx,
\end{align}
where the nuclear charge distribution $m$ has been subjected to an appropriate affine change of variables to account for the affine deformation of the lattice.
Note that the almost-sure existence of this thermodynamic limit has been established for certain random lattices in
\cite{BlancLeBrisLionsStochasticLattice}; see also \cite{BlancLeBrisLionsAtomisticToContinuum,CLBL98} for an overview and related questions. Under the assumptions (A0)-(A3), the almost sure existence of the limit \eqref{DefinitionEeff} could also be shown by an argument similar to our proof of Theorem~\ref{theoremselection}.

In practical computations of the effective energy \eqref{DefinitionEeff}, the infinite-volume average in \eqref{DefinitionEeff} must be replaced by an average over a finite volume, say, a box of the form $[0,L]^d$, an approach also known in the context of continuum mechanics as the \emph{method of representative volumes}. Note that in this setting one must specify appropriate boundary conditions for $\rho$ on $\partial [0,L]^d$.
We shall denote the resulting finite-volume approximation for $E_\infty$ by $E^\RVE_L$.

As boundary layer effects may negatively impact the rate of convergence (in the length $L$) of the representative volume approximations $E^\RVE_L$ towards the thermodynamic limit $E_{\infty}$ (see for instance \cite{FischerVarianceReduction} for a brief discussion of the analogous problem in the context of periodic homogenization of elliptic PDEs), it is desirable to work with periodic representative volumes. In the context of nuclear charge distributions $m$ arising from random lattices, this requires the existence of a \emph{periodization} of the probability distribution of the nuclear charges $m$, that is an $L$-periodic variant $\tilde m$ of the probability distribution of $m$ (see for instance Figure~\ref{FigureSQS} for an illustration). Note that care must be taken to align the definition of the representative volume with a possible underlying periodic structure.
For a more precise explanation of this notion of periodization, see the discussion preceding conditions (A3$_a$)-(A3$_c$) below.
From now on and for the rest of the paper, we will assume that the representative volume approximation $E^\RVE_L$ for the effective energy density $E_\infty$ has been obtained by evaluating the averaged TFW energy (see \eqref{TFWEnergy} below) on such a periodic representative volume.

Our main result --  Theorem~\ref{theoremselection} -- states that the selection approach for representative volumes of Le~Bris, Legoll, and Minvielle \cite{LLM16} increases the accuracy of approximations $E^\RVE_L$, at least for a wide class of random nuclear charge distributions: Instead of choosing a representative volume (that is, an $L$-periodic nuclear charge distribution) uniformly at random from the (periodized) probability distribution, it is better to preselect the representative volume to be ``particularly representative'' for the random alloy in terms of certain basic statistical quantities like the proportion of different types of atoms in the representative volume, the proportion of nearest-neighbor contacts of certain types in the representative volume, and so on. We denote the resulting approximation for the effective energy density by $E^\SQS_L$. In Theorem~\ref{theoremselection} we show that the approximation $E^\SQS_L$ is typically more accurate than the approximation $E^\RVE_L$.
From a mathematical viewpoint, the interest in our main result is twofold:
\begin{itemize}
\item It provides a rigorous justification of the method of ``special quasirandom structures'' in a quantum mechanical model, the setting in which these methods were first developed \cite{ZWFB90}.
\item It provides a first example of a nonlinear PDE for which the selection approach for representative volumes of Le~Bris, Legoll, and Minvielle \cite{LLM16} can be proven to be successful.
\end{itemize}

\begin{figure}
\begin{center}
\begin{tikzpicture}[scale=.4]
\draw[dashed] (0.5,0.5) rectangle (13.5,13.5);
\clip(0.5,0.5) rectangle (13.5,13.5);

\foreach \x in {0,...,14}
\draw (\x,-.5) -- (\x,14.5);

\foreach \y in {0,...,14}
\draw (-.5,\y) -- (14.5,\y);

\foreach \y in {4.5,9.5}
\draw[very thick] (-.5,\y) -- (14.5,\y);

\foreach \x in {4.5,9.5}
\draw[very thick] (\x,-.5) -- (\x,14.5);

\foreach \k/\l in {0/0, 0/5, 0/10, 5/0, 5/5, 5/10, 10/0, 10/5, 10/10}
{
\foreach \x/\y in {\k/\l, \k/2+\l, 1+\k/1+\l, 1+\k/3+\l, 1+\k/4+\l, 2+\k/1+\l, 
				   2+\k/3+\l, 3+\k/2+\l, 3+\k/4+\l, 4+\k/\l, 4+\k/3+\l}
\draw[fill=darkred,color=darkred] (\x,\y) circle (.125);

\foreach \x/\y in {\k/1+\l, 1+\k/\l, 1+\k/2+\l, 2+\k/\l, 2+\k/2+\l, 
				   2+\k/4+\l, 3+\k/\l, 3+\k/1+\l, \k/3+\l, \k/4+\l, 
				   3+\k/3+\l, 4+\k/1+\l, 4+\k/2+\l, 4+\k/4+\l}
\draw[fill=darkblue,color=darkblue] (\x,\y) circle (.125);
}
\end{tikzpicture}
\caption{An illustration of the method of special quasirandom structures: An $L$-periodic ``superlattice'' (with $L\gg 1$) is built to reflect the statistical properties of the random material particularly well -- like the percentage of atoms of the two species, the statistics of nearest-neighbor configurations, the statistics of configurations of three neighboring atoms, and so on. \label{FigureSQS}}
\end{center}
\end{figure}
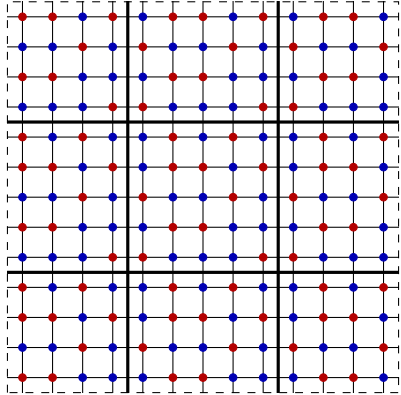

Let us briefly comment on the mechanism for the gain in accuracy achieved by the method of special quasirandom structures. The leading-order contribution to the error in the method of representative volumes consists in fact of fluctuations, while in expectation the method of representative volumes is accurate to much higher order. In fact, in the case of the TFW model the systematic error of the method of representative volumes decays even exponentially in the size of the representative volume
\begin{align*}
\big|\mathbb{E}[E^\RVE_L]-E_\infty\big| \leq C \exp(-c L).
\end{align*}
At the same time, the fluctuations display only CLT scaling behavior
\begin{align*}
\big|E^\RVE_L-\mathbb{E}[E^\RVE_L]\big| \sim L^{-d/2},
\end{align*}
that is they behave like the fluctuations of the average of $L^d$ i.\,i.\,d.\ random variables.
Thus, a variance reduction method -- a method to reduce the fluctuations of the approximations $E^\RVE_L$ while mostly preserving the expected value $\mathbb{E}[E^\RVE_L]$ -- is expected to lead to an increase in accuracy.

The selection of ``particularly representative'' material samples may be viewed as such a \emph{variance reduction method}: In fact, we shall prove that the joint probability distribution of the effective energy $E^\RVE_L$ and statistical quantities like the percentage of atoms of a certain species in the representative volume (and/or quantities like the percentage of nearest-neighbor configurations of two given atomic species, etc.) is close to a multivariate Gaussian. Conditioning on the event that the auxiliary statistical quantity -- which we shall denote by $\mathcal{F}$ -- is close to its expected value then reduces the variance of the computed energies $E^\RVE_L$, provided that $E^\RVE_L$ and the auxiliary quantity are nontrivially correlated (see Figure~\ref{FigureVarianceReduction}). At the same time, the expected value $\mathbb{E}[E^\RVE_L]$ is not changed much by selecting only representative volumes subject to the condition that $\mathcal{F}$ is close to its expected value.

The main challenge in the proof is the derivation of the quantitative multivariate normal approximation result for the joint probability distribution of the energy $E^\RVE_L$ and the statistical quantities $\mathcal{F}$ of the representative volume. Just like in \cite{FischerVarianceReduction}, we make crucial use of the locality properties of these quantities of interest, which allow for a quantitative (multivariate) normal approximation. In \cite{FischerVarianceReduction}, in the context of the homogenization of the linear elliptic PDE $-\nabla \cdot (a_\varepsilon\nabla u)=f$, a localization result for the effective energies
\begin{align*}
a^\RVE_L \xi \cdot \xi := \inf_{v\in H^1_\mathrm{per}([0,L]^d)} \dashint_{[0,L]^d} a(\omega,x)(\xi+\nabla v)\cdot (\xi+\nabla v) \,dx
\end{align*}
has been established: In \cite{FischerVarianceReduction}, the contribution of terms with dependency range $\sim \ell$ to the overall energy $a^\RVE_L \xi \cdot \xi$ is seen to be essentially of the order $\ell^{-d}$, which is essentially twice the order of the fluctuation scaling $\ell^{-d/2}$. By means of a ``multilevel local dependency structure'' \cite{FischerMultilevelLocalDependence}, this allowed for the derivation of a quantitative multivariate normal approximation result for the joint probability distribution of the representative volume approximation $a^\RVE_L$ of the effective coefficient and auxiliary statistical quantities like the averaged coefficient $\mathcal{F}:=\dashint_{[0,L]^d} a \,dx$ \cite{FischerVarianceReduction}.

Due to the strong -- exponential -- localization properties of the TFW model (see \cite{NO17} for the case without point charges and Theorem~\ref{theorempartialijpsi} below for the general case), we in principle would not even need to appeal to the ``multilevel local dependence structure'' introduced in \cite{FischerVarianceReduction,FischerMultilevelLocalDependence}, but could directly work with a multivariate central limit theorem with a standard local dependence structure. However, it will be convenient for us to employ the abstract variance reduction result of Lemma~\ref{lemmavariancereduction}, which is established in \cite{FischerVarianceReduction,FischerMultilevelLocalDependence}.

\begin{figure}
\begin{center}
\includegraphics[scale=0.1]{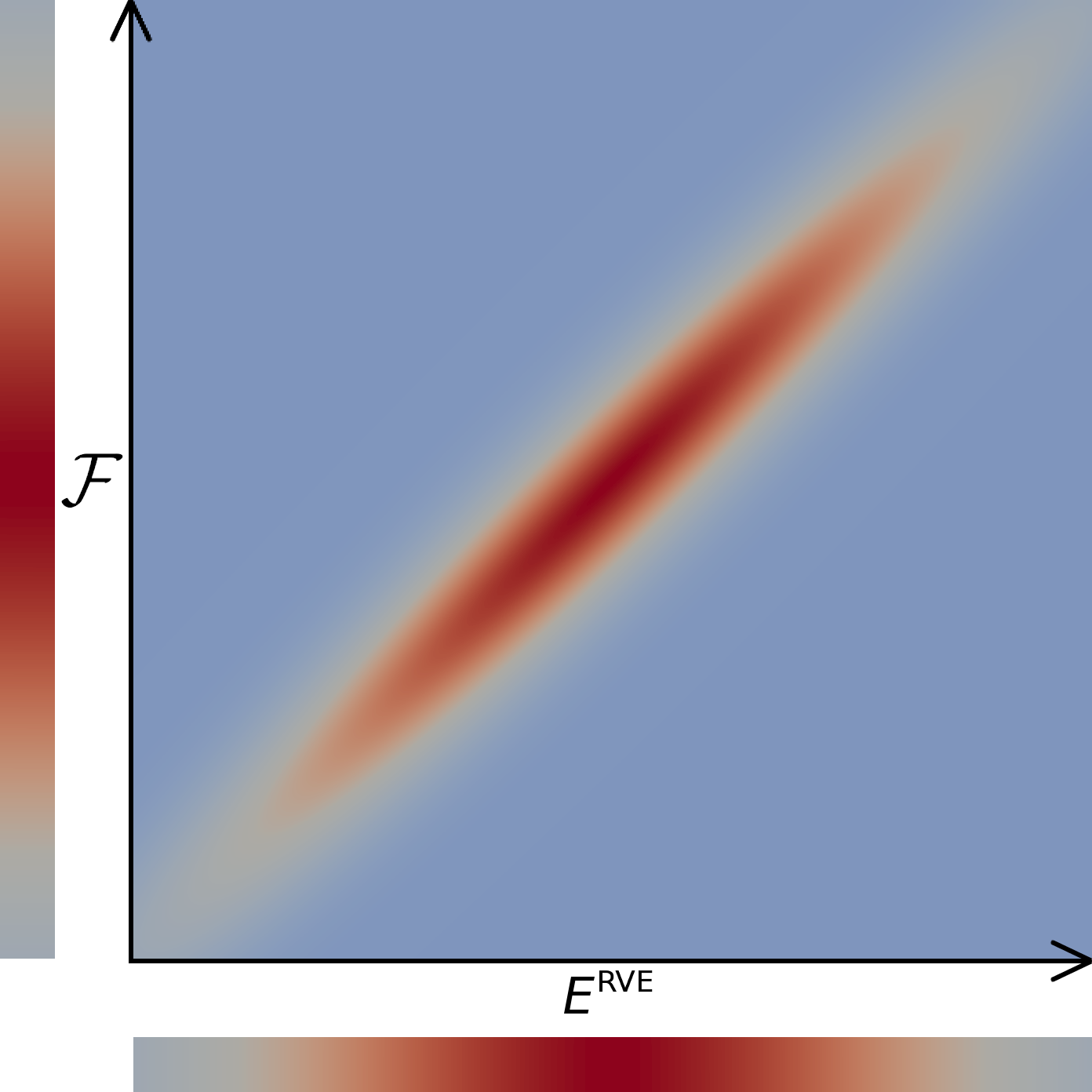}~~~~~
\includegraphics[scale=0.1]{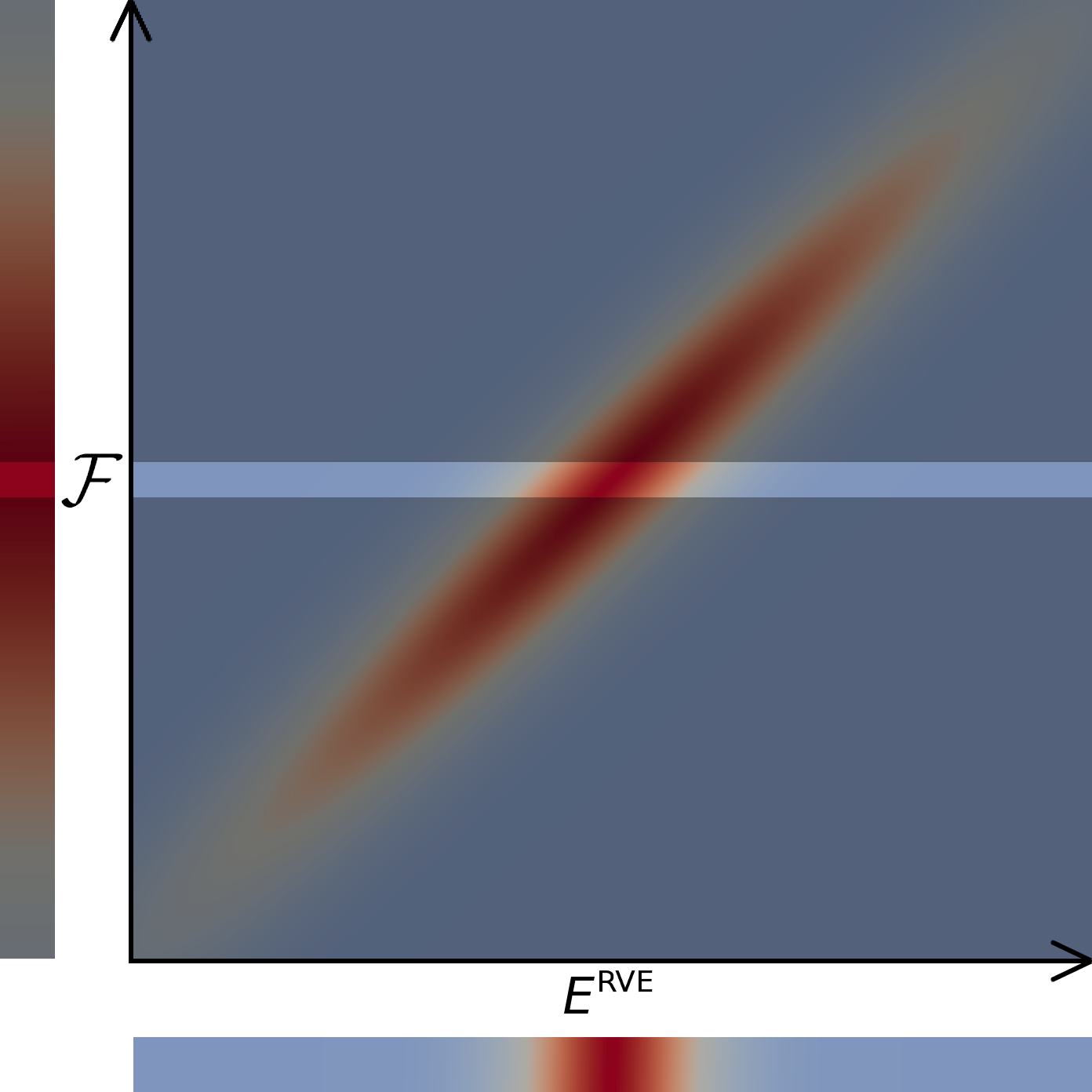}
\caption{The joint probability distribution of the approximations for the effective energy $E^\RVE_L$ and auxiliary statistical quantities $\mathcal{F}$ like the percentage of atoms of a certain species is close to a multivariate Gaussian (left). Conditioning on the auxiliary statistical quantity $\mathcal{F}$ being close to its expected value then reduces the variance of $E^\RVE_L$, provided that the two random variables are nontrivially correlated (right).\label{FigureVarianceReduction}}
\end{center}
\end{figure}

{\bf Notation.}
We use standard notation for Sobolev spaces: By $W^{1,p}(\Omega)$ we denote the space of functions $v\in L^p(\Omega)$ whose distributional derivative $\nabla v$ also belongs to $L^p(\Omega)$, along with the usual norm $||v||_{W^{1,p}(\Omega)}^p = \int_{\Omega}  |v|^p +|\nabla v|^p \,dx$. As usual, we use the abbreviation $H^1(\Omega):=W^{1,2}(\Omega)$. Given $L\geq 1$, by $H^1_\mathrm{per}([0,L]^d)$ we denote the space of $L$-periodic Sobolev functions $v\in H^1([0,L]^d)$.

By $H^1_\mathrm{uloc}(\mathbb{R}^d)$ we denote the space of functions $v:\mathbb{R}^d\rightarrow\mathbb{R}$ whose restrictions $v|_{B_1(x)}$ belong to $H^1(B_1(x))$ for all $x\in \mathbb{R}^d$, with a uniform bound on the local Sobolev norm $||v||_{H^1_{\mathrm{uloc}}(\mathbb{R}^d)}^2 := \sup_{x\in \mathbb{R}^d} \int_{B_1(x)} |v|^2 + |\nabla v|^2 \,dx <\infty$. Similarly, by $L^2_\mathrm{uloc}(\mathbb{R}^d)$ we denote the space of measurable functions $v:\mathbb{R}^d\rightarrow\mathbb{R}$ with finite norm $||v||_{L^2_{\mathrm{uloc}}(\mathbb{R}^d)}^2 := \sup_{x\in \mathbb{R}^d} \int_{B_1(x)} |v|^2 \,dx <\infty$.

By $B_r(x)$ we denote the ball of radius $r$ around $x\in \mathbb{R}^d$. We also use the shorthand notation $B_r:=B_r(0)$.
By $C$ we will denote a generic constant depending only on quantities like $\rho$, $M$, and $\omega_0$ (see the assumptions (A1) and (A3) below), whose precise value may vary from occurrence to occurrence.

For a set $M$, we denote by $\sharp M$ the number of its elements.

For two vector-valued random variables $X$ and $Y$, we denote the covariance matrix as usual by $\Cov[X,Y]$. We also use the notation $\Var X$ as a shorthand notation for $\Cov[X,X]$.

\section{Main Results}
\label{secresults}

In this article, we prove that the selection approach for representative volumes of Le Bris, Legoll, and Minvielle \cite{LLM16} leads to an increase in accuracy when calculating effective energies for random lattices in the context of the Thomas--Fermi--von Weizs\"acker model \eqref{eqtfw}, at least for a wide class of random nuclear charge distributions. For a precise statement of our assumptions and our main result, see (A0)--(A3) and Theorem~\ref{theoremselection} below.

Under more general conditions, we establish an exponential locality result for the TFW model, an auxiliary result that generalizes a corresponding result by Nazar and Ortner \cite{NO17} and that may also be of independent interest.
For a more precise statement of the assumptions and the result, see (A1) and Theorem~\ref{theorempartialijpsi} below.

Consider any Bravais lattice and denote by $F\in \mathbb{R}^{3\times 3}$ a matrix whose columns are given by the corresponding three primitive vectors.
Our key assumptions on the nuclear charge distribution $m$ are as follows.
\begin{enumerate}
\item[(A0)] Let $m$ be a random nuclear charge distribution (a -- random -- locally finite nonnegative Radon measure) on $\mathbb{R}^3$. In other words, let a probability space $(\Omega,\mathcal{F},\mathbb{P})$ be given along with a random variable $m$ taking values in the space of locally finite nonnegative Radon measures on $\mathbb{R}^3$.
\item[(A1)] Suppose that \textit{uniform local finiteness} of the nuclear charge distribution $m$ holds in the following sense: There exist constants $\rho > 0$, $M \geq 0$, and $\omega_0 > 0$ such that $m$ is of the form
\begin{align*}
m = m_c + \sum_{y \in \mathbb P} c_y \delta_y
\end{align*}
for some $m_c=m_c(m) \in L^2_\mathrm{uloc}(\mathbb R^3)$ with $m_c \geq 0$, some  $c_y=c_y(m) > 0$, and some set $\mathbb P=\mathbb P(m) \subset \mathbb R^3$ satisfying $|x-y| \geq 4\rho$ for all $x,y \in \mathbb P$ with $x \neq y$, and in addition the estimate
\begin{align*}
\sup_{x \in \mathbb R^3} \Big( \int_{B_1(x)} m_c^2 \, dy \Big)^\frac12 + \sup_{x \in \mathbb R^3} \Big( \sum_{y \in \mathbb P \cap B_1(x)} c_y^2 \Big)^\frac12 \leq M
\end{align*}
holds.
Furthermore, suppose that an \emph{averaged lower bound} for the nuclear charge density of the form
\begin{align*}
\inf_{x \in \mathbb R^3} \dashint_{B_R(x)} m \, dy \geq \omega_0
\end{align*}
holds for all $R \geq \omega_0^{-1}$ for some $\omega_0>0$.
\item[(A2)] Let $m$ be \textit{stationary}, i.\,e.\ suppose that the law of the shifted charge distribution $m(\cdot + x)$ coincides with the law of $m$ for every $x \in  F\mathbb{Z}^3$.
\item[(A3)] Let $m$ have a \textit{finite range of dependence} $r\geq 1$, i.\,e.\ suppose that for any two Borel sets $A, B \subset \mathbb{R}^3$ with $\dist(A, B) \geq r$ the restrictions $m|_A$ and $m|_B$ are stochastically independent.
\end{enumerate}

We shall also use the concept of a \textit{periodization} of an ensemble of nuclear charge distributions $m$ (where an ensemble of nuclear charge distributions is defined as a probability measure on the space of nuclear charge distributions): A periodization of an ensemble of nuclear charge distributions is an ensemble of nuclear charge distributions $\tilde m$ which are almost surely $L F \mathbb Z^3$-periodic for some $L \gg 1$ and for which the probability distribution of $\tilde m |_{x + F[0, \frac{L}{2}]^3}$ coincides with the probability distribution of $m |_{x + F[0, \frac{L}{2}]^3}$ for all $x \in \mathbb R^3$. Given such a periodization $m$, we substitute (A3) by (A3\textsubscript a) -- (A3\textsubscript c):

\begin{enumerate}
\item[(A3\textsubscript a)] The nuclear charge $\tilde m$ is almost surely $L F \mathbb Z^3$-periodic for some $L \gg 1$.
\item[(A3\textsubscript b)] There exists a \textit{finite range of dependence} $r > 0$ such that for any two Borel $L \mathbb Z^3$-periodic sets $A, B \subset \mathbb R^3$ with $\dist(A, B) \geq r$ the restrictions $\tilde m|_A$ and $\tilde m|_B$ are stochastically independent.
\item[(A3\textsubscript c)] There exists a nuclear charge distribution $m$ satisfying (A1), (A2), and (A3) such that for any $x_0 \in \mathbb R^3$ the law of the restriction $\tilde m|_{x_0 + F[0, \frac{L}{2} ]^3}$ coincides with the law of $m|_{x_0 + F[0, \frac{L}{2} ]^3}$.
\end{enumerate}

Let us briefly comment on our main assumptions. The condition (A1) is nothing but a uniform local upper and lower bound on the charge distribution of the nuclei. The condition (A3) is a strong decorrelation assumption restricting all stochastic dependencies to a scale $r\geq 1$.

The condition (A2) imposes a statistical homogeneity assumption on the random lattice.
Since we want to include the model case of a periodic lattice like $\mathbb{Z}^3$ whose sites are occupied by random atomic nuclei (i.\,e.\ at whose lattice sites there is a random multiple of a Dirac charge; see Figure~\ref{FigureRandomLattice}) in our assumptions, we cannot assume translation invariance of the law of the nuclear charge distribution $m$ with respect to arbitrary shifts $x\in \mathbb{R}^3$. Instead, in the case of the lattice $\mathbb{Z}^3$ we have to restrict the translation invariance to discrete shifts $x\in \mathbb{Z}^3$. As we are interested in the effective elastic properties and as most (elastic) affine deformations of $\mathbb{Z}^3$ destroy the $\mathbb{Z}^3$ periodicity, we have to cover the case of an arbitrary Bravais lattice $F\mathbb{Z}^3$ in our assumption (A2).

Let us now give a precise definition of the TFW energy and its thermodynamic limit.

\begin{definition}
\label{defenergy}
Let $m$ be a nuclear charge distribution satisfying the assumption (A1). For a set $Q \subset \mathbb R^3$ with finite volume, we introduce the \emph{Thomas--Fermi--von Weizs\"acker energy}
\begin{equation}
\label{eqtfwenergy}
E_Q[m] \coleq \int_{Q} |\nabla u|^2  + u^\frac{10}3  + \frac12 (m_c - u^2) \phi \, dx + \sum_{x \in \mathbb P \cap Q} c_{x} (\phi - \phi_{x})(x)
\end{equation}
where $(u,\phi) \in H^1_\mathrm{uloc}(\mathbb R^3) \times L^2_\mathrm{uloc}(\mathbb R^3)$ is the (unique) solution of the TFW equations \eqref{eqtfw} (see Theorem \ref{theoreminf}) and where $\phi_{x} \in L^2(\mathbb R^3)$ is the (decaying) solution of $-\Delta \phi_{x} = c_{x} \delta_{x}$ on $\mathbb R^3$.

We define the thermodynamic limit $E_\infty$ of the energy density -- the \emph{effective energy density} -- as
\begin{align}
\label{eqtfwlimit}
E_\infty \coleq \lim_{L \rightarrow \infty} L^{-3} E_{[0,L]^3}[m]
\end{align}
if the limit exists.

Let $m$ be a nuclear charge distribution satisfying the assumptions (A0)-(A3).
Given $L\geq 1$, let $\tilde m$ be a periodization of the probability distribution of the nuclear charge distribution $m$ subject to (A3\textsubscript a)-(A3\textsubscript c). We define the approximation $E^\RVE_L$ of the effective energy density $E_\infty$ by the representative volume method as
\begin{align*}
E^\RVE_L := \frac{1}{L^3 \det F} &\Bigg( \int_{F[0,L]^3} |\nabla \tilde u|^2  + {\tilde u}^\frac{10}3 + \frac12 (\tilde m_c - \tilde u^2) \tilde \phi \, dx
\\&~~~~~
+ \sum_{x \in {\tilde {\mathbb P}} \cap F[0,L)^3} \tilde c_{x} (\tilde \phi - \phi_{x})(x) \Bigg)
\end{align*}
where $(\tilde u,\tilde \phi) \in H^1_\mathrm{uloc}(\mathbb R^3) \times L^2_\mathrm{uloc}(\mathbb R^3)$ denotes the (unique) solution of the TFW equations \eqref{eqtfw} given the nuclear charge distribution $\tilde m$.
Note that both $\tilde u$ and $\tilde \phi$ inherit the $L$-periodicity of the nuclear charge distribution $\tilde m$ \cite{CLBL98}.

Finally, let $N\in \mathbb{N}$, let $\mathcal{F}$ be a measurable $\mathbb{R}^N$-valued function of the (periodized) nuclear charge distribution $\tilde m$, and let $\delta>0$. We then define $E^\SQS_L$ to denote the approximation of the effective energy density $E_\infty$ using the selection method for representative volumes with the selection performed according to the criterion
\begin{align}
\label{Select}
\big|\mathcal{F}-\mathbb{E}[\mathcal{F}]\big| \leq \delta L^{-3/2}.
\end{align}
In other words, the probability distribution of $E^\SQS_L$ is given as the conditional probability distribution of $E^\RVE_L$ given the event \eqref{Select}.
\end{definition}
Let us briefly discuss the TFW energy \eqref{eqtfwenergy}.
The first two terms in \eqref{eqtfwenergy} correspond to the kinetic energy of the electrons in the Thomas--Fermi--von~Weizs\"acker approximation. The third and fourth term correspond to the Coulomb energy. Here, the contribution from the nuclear charges $m$ has been split into two terms, representing the absolutely continuous part $m_c$ and the singular part $\sum_x c_{x} \delta_{x}$ of the nuclear charge distribution. The presence of the difference $\phi-\phi_x$ in the fourth term in \eqref{eqtfwenergy} corresponds to the usual subtraction of the self-energy of point charges.
Note that the difference $\phi-\phi_x$ satisfies the PDE
\begin{align}
\label{TFWEnergy}
-\Delta (\phi - \phi_{x}) = 4\pi \bigg( m_c - u^2 + \sum_{\genfrac{}{}{0pt}{2}{y \in \mathbb P}{y \neq x}} c_{y} \delta_{y} \bigg),
\end{align}
which by $\phi - \phi_{x} \in H^2(B_{\rho}(x)) \hookrightarrow C^{0,\frac12}(B_{\rho}(x))$ ensures that the pointwise evaluation of $\phi - \phi_{x}$ at the point $x$ in the above definition is indeed meaningful.

We next state additional assumptions and notation which will be needed to formulate our main result on the analysis of the selection approach for representative volumes.
\begin{assumption}
\label{assump}
Consider a probability distribution of nuclear charges $m$ on $\mathbb R^3$ satisfying (A0), (A1), (A2), and (A3). Let $L \in \mathbb{N}$, $L\geq 2$, and assume that there exists an $L$-periodization $\tilde m$ of the probability distribution of $m$ subject to (A1), (A2), and (A3\textsubscript a) -- (A3\textsubscript c).
Let $\mathcal F(\tilde m) = (\mathcal F_1(\tilde m), \dots, \mathcal F_N(\tilde m))$ (for some $N\in \mathbb{N}$) be a collection of statistical quantities of the nuclear charge density $\tilde m$ which are subject to the conditions of Definition \ref{defmultilevel} with $K \leq C_0$ and $B \leq C_0 |\log L|^{C_0}$ for some $C_0 > 0$. Suppose that the covariance matrix of $\mathcal F(\tilde m)$ is nondegenerate and bounded in the natural scaling in the sense 
\begin{equation}
\label{eqvarfmbounds}
L^{-3} \mathrm{Id} \leq \Var \mathcal F(\tilde m) \leq C_0 L^{-3} \mathrm{Id}.
\end{equation}
We introduce the condition number $\kappa$ of the covariance matrix of $(E^\RVE_L, \mathcal F(\tilde m))$
\begin{align*}
\kappa \coleq \kappa \big(\Var (E^\RVE_L, \mathcal F(\tilde m)) \big)
\end{align*}
and the ratio $r_{\mathrm{Var}}$ between the expected order of fluctuations and the actual fluctuations of the approximation $E^\RVE_L$
\begin{align*}
r_{\mathrm{Var}} \coleq \frac{L^{-3}}{\Var E^\RVE_L}.
\end{align*}
\end{assumption}
Let us briefly mention that the following statistical quantities $\mathcal{F}$ satisfy the conditions of Definition~\ref{defmultilevel} below and are therefore admissible choices in our main result (i.\,e.\ in Theorem~\ref{theoremselection} below):
\begin{itemize}
\item The density of nuclei of a specific type
\begin{align*}
\mathcal{F}_{1,a} := \det F^{-1} L^{-3} \sharp\{x\in \tilde {\mathbb{P}}\cap F[0,L)^3:c_x=a\}.
\end{align*}
\item The density of nearest-neighbor contacts of two specified types of nuclei
\begin{align*}
\mathcal{F}_{2,a,b} := \det F^{-1} L^{-3} \sharp\{x\in \tilde {\mathbb{P}}\cap F[0,L)^3:&~c_x=a,~c_{x\pm Fe_j}=b
\\&~~~
\text{ for some }1\leq j\leq 3\}
\end{align*}
in case that the nuclei are arranged on the lattice $F\mathbb{Z}^3$.
\item Similar statistics of configurations of three or more neighboring atoms or corresponding quantities for more general atomic lattices.
\end{itemize}
Note that it is precisely these type of statistics of the random atomic lattice that are considered in the original formulation of the method of special quasirandom structures \cite{ZWFB90}.

We are now in a position to formulate our main result, the gain in accuracy by the selection approach for representative volumes in the context of the TFW model for random alloys. Note that our main result comprises essentially three assertions:
\begin{itemize}
\item The increase in accuracy of $E^\SQS_L$ (as compared to $E^\RVE_L$)  \eqref{eqvariance}, which is achieved via the reduction of fluctuations by essentially the fraction of the variance of $E^\RVE_L$ explained by the statistical quantities $\mathcal{F}$.
\item The higher-order approximation quality \eqref{eqsystematicerror} of the expected value $\mathbb{E}[E^\SQS_L]$.
\item The lower bound \eqref{eqselcritsatisfied} for the probability that a randomly chosen nuclear charge distribution $\tilde m$ meets the selection criterion \eqref{eqselectioncrit}.
\end{itemize}

\begin{theorem}
\label{theoremselection}
Let Assumption \ref{assump} be satisfied. Denote by $E^\RVE_L$ the approximation for the effective energy $E_\infty$ by the standard representative volume element method and by $E^\SQS_L$ the approximation for $E_\infty$ by the selection approach for representative volumes introduced by Le Bris, Legoll, and Minvielle \cite{LLM16} in the case of a representative volume of size $L$. Further assume that in the selection approach, the
representative volumes
are selected from the periodized probability distribution according to the criterion
\begin{equation}
\label{eqselectioncrit}
|\mathcal F(\tilde m) - \mathbb E[\mathcal F(\tilde m)]| \leq \delta L^{-3/2}
\end{equation}
for some $\delta \in (0, 1]$ satisfying $\delta^N \geq C L^{-3/2} |\log L|^C$. Then, the selection approach for representative volumes is subject to the following error analysis:
\begin{enumerate}
\item[(a)] The systematic error of the approximation $E^\SQS_L$ satisfies
\begin{equation}
\label{eqsystematicerror}
\big| \mathbb E \big[ E^\SQS_L \big] - E_\infty \big| \leq \frac{C \kappa^{3/2}}{\delta^N} L^{-3} |\log L|^{C}.
\end{equation}

\item[(b)] The variance of the approximation $E^\SQS_L$ is bounded from above by
\begin{equation}
\label{eqvariance}
\frac{\Var E^\SQS_L}{\Var E^\RVE_L} \leq 1 - (1 - \delta^2) |\rho|^2 + \frac{C \kappa^{3/2} r_{\mathrm{Var}}}{\delta^N} L^{-3/2} |\log L|^{C}
\end{equation}
where $|\rho|^2$ is the fraction of the variance of $E^\RVE_L$ explained by the $\mathcal F(\tilde m)$. In other words, $|\rho|^2$ is the maximal squared correlation coefficient between $E^\RVE_L$ and any linear combination of the $\mathcal F(\tilde m)$. This explained fraction of the variance is given by the expression
\begin{equation}
\label{eqexplainedfraction}
|\rho|^2 \coleq \frac{\Cov [E^\RVE_L, \mathcal F(\tilde m)] (\Var \mathcal F(\tilde m))^{-1} \Cov [\mathcal F(\tilde m), E^\RVE_L]}{\Var E^\RVE_L}.
\end{equation}

\item[(c)] The probability that a randomly chosen nuclear charge distribution $\tilde m$ satisfies the selection criterion \eqref{eqselectioncrit} is at least 
\begin{equation}
\label{eqselcritsatisfied}
\mathbb P[|\mathcal F(\tilde m) - \mathbb E [\mathcal F(\tilde m)]| \leq \delta L^{-3/2}] \geq c(N) \delta^N.
\end{equation}
\end{enumerate}
\end{theorem}

We next state our precise assumption on the statistical quantities $\mathcal{F}$.

\begin{definition}[similar to {\cite[Definition~3]{FischerMultilevelLocalDependence}}]
\label{defmultilevel}
Let $L \geq 2$, and consider a probability distribution of $LF\mathbb{Z}^3$-periodic nuclear charges $\tilde m$ on $\mathbb R^3$ satisfying (A1), (A2), and (A3\textsubscript a) -- (A3\textsubscript c). Let $X = X[\tilde m]$ be a random variable of the periodized nuclear charge.
We say that $X$ is a sum of random variables with multilevel local dependence if there exist random variables $X_y^n = X_y^n[\tilde m]$, $0 \leq n \leq 1 + \log_2 L$, $y \in 2^n F\mathbb Z^3 \cap F[0,L)^3$, and constants $K\geq 1$ and $B \geq 1$ with the following properties:
\begin{itemize}
\item The random variable $X_y^n[\tilde m]$ only depends on $\tilde m|_{y + K \log_2 L F [-2^n, 2^n]^3}$.
\item We have 
\begin{align*}
X = \sum_{n=0}^{1 + \log_2 L} \sum_{y \in 2^n F \mathbb Z^3 \cap F [0,L)^3} X_y^n.
\end{align*}
\item The random variables satisfy almost surely
\begin{align*}
| X_y^n | \leq B L^{-3}.
\end{align*}
\end{itemize}
\end{definition}

Our proof of Theorem~\ref{theoremselection} makes use of the following exponential locality result for solutions to the TFW equations, which extends a similar locality result of \cite{NO17} to include point charges and which may be of independent interest.
\begin{theorem}
\label{theorempartialijpsi}
Let $m_1$ and $m_2$ be two nonnegative nuclear charge distributions (i.\,e.\ nonnegative locally finite Radon measures) subject to the assumption (A1). Denote by $(u_1,\phi_1)\in \Huloc \times \Luloc$ and $(u_2,\phi_2)\in \Huloc \times \Luloc$ the corresponding solutions to the TFW equations \eqref{eqtfw}. Then the perturbations of the electronic density $w := u_1 - u_2$ and the potential $\psi := \phi_1 - \phi_2$ decay exponentially away from the perturbation of the nuclear charge distribution $\delta m := m_1 - m_2$. More precisely, there exist constants $C = C(\rho, M, \omega_0) > 0$ and $\gamma = \gamma(\rho, M, \omega_0) > 0$ such that for all $y \in \mathbb R^3$
the estimate
\begin{align*}
&\int_{\mathbb R^3} \Big( w^2 + |\nabla w|^2 + \psi^2 + \eta |\nabla \psi|^2 + \eta^2 \sum_{i,j=1}^{3} |\partial_{ij} \psi|^2 \Big) e^{-2\gamma|x-y|} \, dx
\\&~~~~
\leq C \left( \int_{\{\eta < 1\}} (w^2 + \psi^2) e^{-2\gamma |x-y|} \, dx + \int_{\mathbb R^3} (\delta m_c)^2 e^{-2\gamma |x-y|} \, dx \right)
\end{align*}
holds, where the cutoff $\eta$ ($0\leq \eta\leq 1$) is defined in Assumption \ref{assumplocality} below and where $\delta m_c:=m_{1,c}-m_{2,c}$. Here, we have made use of the decomposition $m_i=m_{i,c}+\sum_{x\in \mathbb P_i} c_{i,x} \delta_x$ provided by assumption (A1).
\end{theorem}

The following proposition comprises the exponential locality result for the TFW energy. Given two nuclear charge distributions $m_1$ and $m_2$, the difference between the values of the TFW energy evaluated for $m_1$ and $m_2$ within the domain $Q_1$ exponentially decreases with the distance between $\supp (m_1 - m_2)$ and $Q_1$.

\begin{corollary}
\label{prop2diff}
Let Assumption \ref{assumplocality} be satisfied. Then, there exist constants $C = C(\rho, M, \omega_0) > 0$ and $c = c(\rho, M, \omega_0) > 0$ such that for any cube $Q_1 \subset \mathbb R^3$ with unit volume the estimate
\begin{equation}
\label{eq2diff}
\big| E_{Q_1}[m_1] - E_{Q_1}[m_2] \big| \leq C e^{-c \, \dist(\supp(m_1 - m_2), \, Q_1)}
\end{equation}
holds true.
\end{corollary}

\section{Analysis of the Method of Special Quasirandom Structures}
\label{secproofvariancereduction}

The following lemma serves as the main technical tool to prove Theorem \ref{theoremselection}. In fact, it is an abstract version of \cite[Theorem 2]{FischerVarianceReduction}: One may adapt the proof of \cite[Theorem 2]{FischerVarianceReduction} in a one-to-one fashion to establish Lemma \ref{lemmavariancereduction}.

\begin{lemma}
\label{lemmavariancereduction}
Let $d, N \in \mathbb N$, $d \geq 2$, $N \geq 1$, $C_0 \geq 0$, $L \geq 2$, and let $C > 0$ denote a generic constant which only depends on $d$, $N$, and $C_0$ as well as on $K$ and $B$ from Definition \ref{defmultilevel}. Let $Z = (Z_0,Z_1,\dots,Z_N)$ be a vector of random variables. Suppose that each $Z_i$, $i \in \{0, \dots, N\}$, is a sum of random variables with multilevel local dependence according to Definition \ref{defmultilevel}. Assume that the covariance matrix of $(Z_1, \dots, Z_N)$ is nondegenerate and bounded in the natural scaling
in the sense that 
\begin{equation*}
L^{-d} \mathrm{Id} \leq \Var (Z_1, \dots, Z_N) \leq C_0 L^{-d} \mathrm{Id}.
\end{equation*}

Let $\delta \in (0, 1]$ satisfy $\delta^N \geq C L^{-d/2} |\log L|^C$, and 
let $Z_{0, \mathrm{sel}}$ be a random variable whose law coincides with the probability distribution of the random variable $Z_0$ conditioned on the event $|Z_i - \mathbb E[Z_i]| \leq \delta L^{-d/2}$ for all $i \in \{1,\dots,N\}$. 

Introduce the condition number $\kappa$ of the covariance matrix of $(Z_0, \dots, Z_N)$,
\[
\kappa \coleq \kappa \big(\Var (Z_0, \dots, Z_N) \big),
\]
and the ratio $r_{\mathrm{Var}}$ between the expected order of fluctuations and the actual fluctuations of $Z_0$
\[
r_{\mathrm{Var}} \coleq \frac{L^{-d}}{\Var Z_0}.
\]
Then, the following estimates hold true:
\begin{enumerate}
\item[(a)] The difference of the expected values of $Z_{0,\mathrm{sel}}$ and $Z_0$ satisfies
\begin{equation*}
\big| \mathbb E \big[ Z_{0,\mathrm{sel}} \big] - \mathbb E \big[ Z_0 \big] \big| \leq \frac{C \kappa^{3/2}}{\delta^N} L^{-d} |\log L|^{C}.
\end{equation*}

\item[(b)] The variance of $Z_{0,\mathrm{sel}}$ is bounded from above by
\begin{equation*}
\frac{\Var Z_{0,\mathrm{sel}}}{\Var Z_{0}} \leq 1 - (1 - \delta^2) |\rho|^2 + \frac{C \kappa^{3/2} r_{\mathrm{Var}}}{\delta^N} L^{-d/2} |\log L|^{C}
\end{equation*}
where $|\rho|^2$ is the fraction of the variance of $Z_{0}$ explained by $(Z_1, \dots, Z_N)$. In other words, $|\rho|^2$ is the maximal squared correlation coefficient between $Z_{0}$ and any linear combination of the $Z_i$, $i \in \{1, \dots, N\}$. This fraction of the variance of $Z_0$ explained by $Z_1$, \ldots, $Z_N$ is given by the expression
\begin{equation*}
~~~~~~~~~~|\rho|^2 \coleq \frac{\Cov [Z_0, (Z_1, \dots, Z_N)] (\Var (Z_1, \dots, Z_N))^{-1} \Cov [(Z_1, \dots, Z_N), Z_0]}{\Var Z_0}.
\end{equation*}

\item[(c)] The probability that $(Z_1, \dots, Z_N)$ satisfies the imposed selection criterion is at least 
\begin{equation*}
\mathbb P[|(Z_1, \dots, Z_N) - \mathbb E [(Z_1, \dots, Z_N)]| \leq \delta L^{-d/2}] \geq c(N) \delta^N.
\end{equation*}
\end{enumerate}
\end{lemma}

By combining the locality properties of the TFW model established in Theorem~\ref{theorempartialijpsi} with the abstract variance reduction result of Lemma~\ref{lemmavariancereduction}, we now establish our main result.

\begin{proof}[Proof of Theorem \ref{theoremselection}]
Throughout the proof, we will assume $F=\Id$. The case of general $F$ is similar.

The proof of Theorem~\ref{theoremselection} is an immediate consequence of Lemma~\ref{lemmavariancereduction} as soon as we have established two results:
\begin{itemize}
\item[1)] The approximation $E^\RVE_L$ for the thermodynamic limit energy $E_\infty$ by the method of representative volumes
\begin{align*}
~~~~~~~~~~
&E^\RVE_L
:= L^{-3} E_{[0,L]^3}[\tilde m]
\\&
= L^{-3} \int_{[0,L]^3} |\nabla \tilde u|^2  + \tilde u^\frac{10}3  + \frac12 (\tilde m_c - \tilde u^2) \tilde \phi \, dx + L^{-3} \sum_{x \in \mathbb P \cap [0,L)^3} c_{x} (\tilde \phi - \phi_{x})(x)
\end{align*}
(where $(\tilde u,\tilde \phi)$ denotes the solution of the TFW equations as stated in Theorem~\ref{theoreminf}\,; note that by the $L$-periodicity of $\tilde m$, the solution $(\tilde u,\tilde \phi)$ is $L$-periodic) is a sum of random variables with multilevel local dependence according to Definition~\ref{defmultilevel}.
\item[2)] The estimate for the systematic error
\begin{align}
\label{ErrorEstimateSystmatic}
|\mathbb{E}[E^\RVE_L]-E_\infty|\leq C \exp(-c L)
\end{align}
holds.
\end{itemize}

\emph{Step 1: Proof of 1).}
We denote by $Q_\ell(x)$ the cube $x+[-\frac{\ell}{2},\frac{\ell}{2})^3$.
In order to establish the property 1), we may write
\begin{align*}
E^\RVE_L = \sum_{y \in \mathbb Z^3 \cap [0,L)^3} E_{y}^{0} + E^{1+\log_2 L}
\end{align*}
with
\begin{subequations}
\label{ChoiceEyn}
\begin{align}
E_{y}^{0} &:= L^{-3} E_{Q_1(y_i)}[\tilde m|_{Q_{K \log_2 L}(y)}^\mathrm{ext}]
\end{align}
and
\begin{align}
E^{1 + \log_2 L} &:=
L^{-3} \sum_{y \in 2^n \mathbb Z^3 \cap [0,L)^3} \big( E_{Q_1(y)}[\tilde m] - E_{Q_1(y)}[\tilde m|_{Q_{K \log_2 L}(y)}^\mathrm{ext}] \big).
\end{align}
\end{subequations}
Here, we employ the notation $\tilde m|_{Q_{K \log_2 L}(y_i)}^\mathrm{ext}$ to denote the extension of the restriction $\tilde m|_{Q_{K \log_2 L}(y_i)}$ to $\mathbb{R}^3$ by a constant multiple of the Lebesgue measure
\begin{align*}
\tilde m|_{Q_{K \log_2 L}(y_i)}^\mathrm{ext}(A) := \tilde m(A\cap Q_{K \log_2 L}(y_i)) + \int_{A\setminus Q_{K \log_2 L}(y_i)} 1 \,dx
\end{align*}
for any Borel set $A\subset \mathbb{R}^3$. The constant $K$ will be chosen below. Note that $\tilde m|_{Q_{K \log_2 L}(y_i)}^\mathrm{ext}$ is still subject to uniform bounds of the form (A1).

The first of the three conditions on the $X_y^n$ in Definition \ref{defmultilevel} is satisfied for the choice \eqref{ChoiceEyn} as the random variables $E_{y_i}^{0}$ only depend on $\tilde m|_{Q_{K \log_2 L}(y_i)}$. The second condition trivially holds true due to the definition of $E_{y_i}^{0}$ and $E^{1 + \log_2 L}$. The third condition for the $E_{y_i}^{0}$ -- that is, the bound $|E_{y_i}^{0}| \leq B L^{-3}$ -- follows from the structure of the Thomas--Fermi--von Weizs\"acker energy in \eqref{eqtfwenergy} and the bounds on $u$, $\phi$ and $m$ from Proposition \ref{propbounds}. Finally, to establish 1) it only remains to show the bound $|E^{1 + \log_2 L}|\leq C L^{-3}$. As a consequence of Corollary~\ref{prop2diff} and the equality $\tilde m=\tilde m|_{Q_{K \log_2 L}(y_i)}^\mathrm{ext}$ on $Q_{K \log_2 L}(y_i)$, we derive
\begin{align*}
\Big| E_{Q_1(y_i)}[\tilde m] - E_{Q_1(y_i)}[\tilde m|_{Q_{K \log_2 L}(y_i)}^\mathrm{ext}] \Big| \leq C e^{-c \frac{K \log_2 L - 1}{2}} \leq C e^{-c \frac{(K-1) \log_2 L}{2}} \leq C L^{-3}
\end{align*}
for the choice $K \coleq \frac{6}{c} + 1$, where the positive constants $C$ and $c$ only depend on $\rho$, $M$, and $\omega_0$. Taking the sum over all $y\in \mathbb{Z}^3 \cap [0,L)^d$ and multiplying by $L^{-3}$, we have shown the desired bound $|E^{1 + \log_2 L}| \leq C L^{-3}$.

\emph{Step 2: Proof of 2).}
In order to establish \eqref{ErrorEstimateSystmatic}, we may write
\begin{align*}
E^\RVE_L
&= L^{-3} \sum_{y\in \frac{L}{4} \{0,1,2,3\}^3} E_{Q_{L/4}(y)}[\tilde m]
\\&
= L^{-3} \sum_{y\in \frac{L}{4} \{0,1,2,3\}^3} E_{Q_{L/4}(y)}[\tilde m|_{Q_{L/2}(y)}^\mathrm{ext}]
\\&~~~~
+L^{-3} \sum_{y\in \frac{L}{4} \{0,1,2,3\}^3} \big(E_{Q_{L/4}(y)}[\tilde m]-E_{Q_{L/4}(y)}[\tilde m|_{Q_{L/2}(y)}^\mathrm{ext}]\big).
\end{align*}
Taking the expectation and estimating the terms in the last line using Corollary~\ref{prop2diff} and the equality $\tilde m=\tilde m|_{Q_{L/2}(y)}^\mathrm{ext}$ on $Q_{L/2}(y)$, we obtain
\begin{align}
\label{RestrictionError}
\bigg|\mathbb{E}[E^\RVE_L]
-L^{-3} \sum_{y\in \frac{L}{4} \{0,1,2,3\}^3} \mathbb{E}\big[E_{Q_{L/4}(y)}[\tilde m|_{Q_{L/2}(y)}^\mathrm{ext}]\big]
\bigg|
\leq C \exp(-cL).
\end{align}
Using the equality in law of $\tilde m|_{Q_{L/2}(y)}^\mathrm{ext}$ and $m|_{Q_{L/2}(y)}^\mathrm{ext}$, we deduce
\begin{align*}
\bigg|\mathbb{E}[E^\RVE_L]
-L^{-3} \sum_{y\in \frac{L}{4} \{0,1,2,3\}^3} \mathbb{E}\big[E_{Q_{L/4}(y)}[m|_{Q_{L/2}(y)}^\mathrm{ext}]\big]
\bigg|
\leq C \exp(-cL).
\end{align*}
Deriving the same estimate as in \eqref{RestrictionError} for $m$, we obtain \eqref{ErrorEstimateSystmatic}.
\end{proof}

\section{Locality of the TFW Equations Involving Point Charges}
\label{seclocality}

An important tool which we will utilize frequently to deal with the Dirac charges and the corresponding singularities of the potential is given by the class of cut-off functions $\eta_\rho$ which we introduce now. Note that given a nuclear charge distribution $m$, we will define $\eta_\rho$ in such a way that $\eta_\rho$ vanishes in a $\rho$-neighborhood of the point charges and that $\eta_\rho\equiv 1$ holds outside of a $2\rho$-neighborhood of all point charges.

\begin{notation}
\label{noteta}
For $\rho > 0$, we define the cut-off function $\wt\eta_\rho : [0, \infty) \rightarrow [0, 1]$ via 
\[
\wt\eta_\rho(r) \coleq \exp \left( -\frac{\rho \log 2}{2 (r - \rho)} \right)
\]
for $r \in (\rho, \frac32\rho]$, $\wt\eta_\rho(r) \coleq 1 - \wt\eta_\rho(3\rho-r)$ for $r \in (\frac32\rho, 2\rho)$, $\wt\eta_\rho = 0$ on $[0, \rho]$ and $\wt\eta_\rho = 1$ on $[2\rho, \infty)$.
For all $\rho > 0$, one thus has $\wt\eta_\rho \in C^1([0, \infty))$ and there exists a constant $c_\eta(\rho) > 0$ such that $\frac{|\wt\eta_\rho'|^2}{\wt\eta_\rho} \leq c_\eta(\rho)$ holds true on $(\rho,\infty)$. 

Moreover, for any discrete set $\mathbb P \subset \mathbb R^3$ satisfying $|x - y| \geq 4\rho$ for some $\rho > 0$ and all $x,y \in \mathbb P$, $x \neq y$, define $\eta_\rho : \mathbb R^3 \rightarrow [0,1]$ via $\eta_\rho \coleq \wt\eta_\rho(|\cdot - \, z|)$ on $B_{2\rho}(z)$ for all $z \in \mathbb P$ and $\eta_\rho \coleq 1$ elsewhere. Then, we have $\eta_\rho \in C^1(\mathbb R^3)$ and 
\begin{equation}
\label{eqestimateeta}
\frac{|\nabla \eta_\rho|^2}{\eta_\rho} \leq c_\eta(\rho)
\end{equation}
is valid on $\{\eta_\rho > 0\}$.
\end{notation}

We now collect the set of assumptions and notations which we employ within the subsequent lemmas and Theorem \ref{theorempartialijpsi}.

\begin{assumption}
\label{assumplocality}
Let $m_i$, $i \in \{1,2\}$, be charge distributions satisfying (A1), (A2), and (A3), 
and let $(u_i, \phi_i) \in H^1_\mathrm{uloc}(\mathbb R^3) \times L^2_\mathrm{uloc}(\mathbb R^3)$ denote the unique weak solution to the Thomas--Fermi--von Weizs\"acker equations 
\begin{equation}
\label{eqtfi}
\begin{cases}
-\Delta u_i + \frac53 u_i^\frac73 - \phi_i u_i = 0, \\
-\Delta \phi_i = 4\pi (m_i - u_i^2),
\end{cases}
\end{equation}
(see Theorem \ref{theoreminf}).

We define the short-hand notations $w \coleq u_1 - u_2$, $\psi \coleq \phi_1 - \phi_2$, $\delta m \coleq m_1 - m_2$, and $\delta m_c \coleq m_{c,1} - m_{c,2}$. The measure $\delta m$ then may be decomposed as
\begin{align*}
\delta m = \delta m_c + \sum_{x \in \mathbb P'} \delta c_x \delta_{x}
\end{align*}
where $\mathbb P' \subset \mathbb P_1\cup \mathbb P_2$ is the set of all $x \in \mathbb P_1\cup \mathbb P_2$ for which $\delta c_x \coleq c_{1,x} - c_{2,x} \neq 0$ holds true. Moreover, we will use the notation $\eta$ to denote the cutoff function $\eta_\rho$ from Notation \ref{noteta} with $\mathbb P'$ taking the role of $\mathbb P$. Finally, we introduce $\xi \coleq e^{-\gamma | \cdot - \, y |}$ for some $0 < \gamma < 1$ and $y \in \mathbb R^3$. Note that this choice entails $|\nabla \xi| \leq \gamma \xi \leq \xi$.
\end{assumption}

As a key step towards Theorem \ref{theorempartialijpsi}, we derive an upper bound for the weighted $L^2$-norm of $w$, $\nabla w$, and $\sqrt{\eta} \nabla \psi$ by adapting the strategy in \cite{NO17} to the more general case of locally finite nonnegative Radon measures representing the nuclear charges.

\begin{lemma}
\label{lemmaw}
Let Assumption \ref{assumplocality} be satisfied. Then, there exist positive constants $C = C(\rho, M, \omega_0) > 0$ and $\gamma = \gamma(\rho, M, \omega_0) > 0$ such that
\vspace*{-1ex}
\begin{multline*}
\int_{\mathbb R^3} (w^2 + |\nabla w|^2 + \eta |\nabla \psi|^2) \xi^2 \, dx \\
\leq C \bigg( \int_{\mathbb R^3} (w^2 + \psi^2) |\nabla \xi|^2 \, dx + \int_{\{\eta < 1\}} (w^2 + \psi^2) \xi^2 \, dx + \int_{\mathbb R^3} \big|\delta m_c \, \psi\big| \xi^2 \, dx \bigg)
\end{multline*}
holds, where $\eta$ and $\xi$ (depending on $\gamma$) are defined in Assumption \ref{assumplocality}.
\end{lemma}

\begin{proof}
Following the argumentation in \cite{NO17}, we have
\begin{align}
-\Delta w &= \frac53 \Big(u_2^\frac73 - u_1^\frac73\Big) + \phi_1 u_1 - \phi_2 u_2, \label{eqw} \\
-\Delta \psi &= 4\pi (u_2^2 - u_1^2) + 4\pi \,\delta m, \label{eqpsi}
\end{align}
and test \eqref{eqw} with $w\xi^2$ (note that by $w \in H^1_{\textrm{uloc}}(\mathbb R^3)$ and the exponential decay of $\xi$ and $\nabla \xi$, $w\xi^2$ is indeed an admissible test function). This yields 
\[
\int_{\mathbb R^3} \nabla w \cdot \nabla (w \xi^2) \, dx + \frac53 \int_{\mathbb R^3} \Big(u_1^\frac73 - u_2^\frac73\Big) w\xi^2 \, dx - \int_{\mathbb R^3} (\phi_1 u_1 - \phi_2 u_2) w\xi^2 \, dx = 0.
\]
The elementary estimate 
\[
\big(u_1^\frac73 - u_2^\frac73\big) (u_1 - u_2) = \Big(u_1^\frac43 + u_2^\frac43\Big) w^2 + u_1 u_2 (u_1^\frac13 - u_2^\frac13) w \geq \frac12 \Big(u_1^\frac43 + u_2^\frac43\Big) w^2 + \nu w^2
\]
with $\nu \coleq \frac12 \inf_{\mathbb R^3} \big( u_1^\frac43 + u_2^\frac43 \big) > 0$ only depending on $\rho$, $M$, and $\omega_0$, as well as the identities
\begin{align*}
\phi_1 u_1 - \phi_2 u_2 &= \frac{\phi_1 + \phi_2}{2} w + \frac{u_1 + u_2}{2} \psi \\
\nabla w \cdot \nabla (w\xi^2) &= |\nabla(w\xi)|^2 - w^2 |\nabla \xi|^2
\end{align*}
give rise to 
\begin{align*}
&\int_{\mathbb R^3} |\nabla(w\xi)|^2 \, dx + \frac56 \int_{\mathbb R^3} \Big(u_1^\frac43 + u_2^\frac43\Big) w^2 \xi^2 \, dx
\\&
- \frac12 \int_{\mathbb R^3}  (\phi_1 + \phi_2) w^2 \xi^2 \, dx + \nu \int_{\mathbb R^3} w^2 \xi^2 \, dx
\\&
\leq \int_{\mathbb R^3} w^2 |\nabla \xi|^2 \, dx + \frac12 \int_{\mathbb R^3} (u_1 + u_2) \psi w \xi^2 \, dx.
\end{align*}
Now, consider the operators $L_i \coleq -\Delta + \frac53 u_i^\frac43 - \phi_i$ for $i \in \{1,2\}$ and 
\begin{equation}
\label{eqdefl}
L \coleq -\Delta + a, \qquad a \coleq \frac56 \Big( u_1^\frac43 + u_2^\frac43 \Big) - \frac12 (\phi_1 + \phi_2) \in L^2_\mathrm{uloc}(\mathbb R^3)
\end{equation}
(the latter inclusion holding by $u_i\in H^2_\mathrm{uloc}(\mathbb R^3)$).
Due to Lemma \ref{lemmapositiveoperator}, $L_1$, $L_2$, and hence $L$ are non-negative operators on $H^1(\mathbb R^3)$. 
In fact, $w\xi \in H^1(\mathbb R^3)$, $\langle w \xi, L(w \xi) \rangle \geq 0$ and 
\begin{align*}
\langle w \xi, L(w \xi) \rangle + \nu \int_{\mathbb R^3} w^2 \xi^2 \, dx \leq \int_{\mathbb R^3} w^2 |\nabla \xi|^2 \, dx + \frac12 \int_{\mathbb R^3} (u_1 + u_2) \psi w \xi^2 \, dx.
\end{align*}
We continue by testing \eqref{eqpsi} with $\eta \psi \xi^2 \in H^1(\mathbb R^3)$, which results in
\begin{align*}
\int_{\mathbb R^3} \nabla \psi \cdot \nabla (\eta \psi \xi^2) = -4 \pi \int_{\mathbb R^3} \eta (u_1 + u_2) \psi w \xi^2 \, dx + 4 \pi \int_{\mathbb R^3} \eta \, \delta m_c \, \psi \xi^2 \, dx.
\end{align*}
The representation 
\begin{align*}
\nabla \psi \cdot \nabla (\eta \psi \xi^2) = \eta |\nabla(\psi \xi)|^2 - \eta \psi^2 |\nabla \xi|^2 + \psi \xi^2 \nabla \psi \cdot \nabla \eta
\end{align*}
leads one to
\begin{multline*}
\frac12 \int_{\mathbb R^3} \eta (u_1 + u_2) \psi w \xi^2 \, dx = -\frac{1}{8\pi} \int_{\mathbb R^3} \nabla \psi \cdot \nabla (\eta \psi \xi^2) \, dx + \frac12 \int_{\mathbb R^3} \eta \, \delta m_c \, \psi \xi^2 \, dx \\
= -\frac{1}{8\pi} \int_{\mathbb R^3} \eta |\nabla(\psi \xi)|^2 \, dx + \frac{1}{8\pi} \int_{\mathbb R^3} \eta \psi^2 |\nabla \xi|^2 \, dx \\
- \frac{1}{8\pi} \int_{\mathbb R^3} \psi \xi^2 \nabla \psi \cdot \nabla \eta \, dx + \frac12 \int_{\mathbb R^3} \eta \, \delta m_c \, \psi \xi^2 \, dx
\end{multline*}
and, thus, 
\begin{multline*}
\langle w \xi, L(w \xi) \rangle + \nu \int_{\mathbb R^3} w^2 \xi^2 \, dx + \frac{1}{8\pi} \int_{\mathbb R^3} \eta |\nabla(\psi \xi)|^2 \, dx \\
\leq \int_{\mathbb R^3} w^2 |\nabla \xi|^2 \, dx + \frac{1}{8\pi} \int_{\mathbb R^3} \psi^2 |\nabla \xi|^2 \, dx - \frac{1}{8\pi} \int_{\mathbb R^3} \psi \xi^2 \nabla \psi \cdot \nabla \eta \, dx \\
+ \frac12 \int_{\mathbb R^3} (1 - \eta) (u_1 + u_2) \psi w \xi^2 \, dx + \frac12 \int_{\mathbb R^3} \eta \, \delta m_c \, \psi \xi^2 \, dx.
\end{multline*}
By estimating $\eta |\nabla \psi|^2 \xi^2 \leq 2 \eta |\nabla(\psi \xi)|^2 + 2 \eta \psi^2 |\nabla \xi|^2$, by applying an absorption argument together with $|\nabla \eta|^2 \leq c_\eta(\rho) \eta$ to $\nabla \psi$ on the right hand side, and by employing the uniform $L^\infty$-bound for $u_1$ and $u_2$ from Proposition \ref{propbounds}, we obtain
\begin{multline}
\label{eqboundw2}
\langle w \xi, L(w \xi) \rangle + \int_{\mathbb R^3} (w^2 + \eta |\nabla \psi|^2) \xi^2 \, dx \\
\leq C \bigg( \int_{\mathbb R^3} (w^2 + \psi^2) |\nabla \xi|^2 \, dx + \int_{\{\eta < 1\}} (w^2 + \psi^2) \xi^2 \, dx + \int_{\mathbb R^3} |\delta m_c \, \psi| \xi^2 \, dx \bigg)
\end{multline}
As $L = -\Delta + a$ with $a \in L^2_{\mathrm{uloc}}(\mathbb R^3)$ defined in \eqref{eqdefl} and $|\nabla w|^2 \xi^2 \leq 2 |\nabla(w \xi)|^2 + 2 w^2 |\nabla \xi|^2$, we derive
\begin{multline}
\label{eqboundnablaw2}
\int_{\mathbb R^3} |\nabla w|^2 \xi^2 \, dx + \int_{\mathbb R^3} (w^2 + \eta |\nabla \psi|^2) \xi^2 \, dx \leq 2 \int_{\mathbb R^3} |a|w^2\xi^2 \, dx \\
+ C \bigg( \int_{\mathbb R^3} (w^2 + \psi^2) |\nabla \xi|^2 \, dx + \int_{\{\eta < 1\}} (w^2 + \psi^2) \xi^2 \, dx + \int_{\mathbb R^3} |\delta m_c \, \psi | \xi^2 \, dx \bigg).
\end{multline}
In order to handle the $|a|w^2\xi^2$-term, we bound the integral over $\mathbb R^3$ by the sum over all integrals over all balls of radius $1$ located at points with integer coordinates. After applying a Gagliardo--Nirenberg-estimate and Young's inequality, we again arrive at (a multiple of) an integral over $\mathbb R^3$ as every $x \in \mathbb R^3$ can belong to at most eight unit balls around integer points:
\begin{align*}
\int_{\mathbb R^3} &|a|w^2\xi^2 \, dx \leq \sum_{x \in \mathbb Z^3} \| a \|_{L^2(B_1(x))} \| w \xi \|_{L^4(B_1(x))}^2 \\ 
&\leq \sum_{x \in \mathbb Z^3} C \| a \|_{L^2_\mathrm{uloc}(\mathbb R^3)} \| w \xi \|_{L^2(B_1(x))}^\frac12 \| w \xi \|_{H^1(B_1(x))}^\frac32 \\ 
&\leq \sum_{x \in \mathbb Z^3} \Big( C(\tau) \| a \|_{L^2_\mathrm{uloc}(\mathbb R^3)}^4 \| w \xi \|_{L^2(B_1(x))}^2 + \frac{\tau}{8} \| w \xi \|_{H^1(B_1(x))}^2 \Big) \\
&\leq \Big( C(\tau) \| a \|_{L^2_\mathrm{uloc}(\mathbb R^3)}^4 + \tau \Big) \int_{\mathbb R^3} w^2 \xi^2 \, dx + \tau \int_{\mathbb R^3} |\nabla w|^2 \xi^2 \, dx + \tau \int_{\mathbb R^3} w^2 |\nabla \xi|^2 \, dx.
\end{align*}
We now choose $\tau > 0$ --- arising from Young's inequality --- in such a way that $\int_{\mathbb R^3} |\nabla w|^2 \xi^2 \, dx$ can be absorbed on the left hand side of \eqref{eqboundnablaw2}. As $L$ is non-negative, the right hand side of \eqref{eqboundw2} already serves as an upper bound for $\int_{\mathbb R^3} w^2 \xi^2 \, dx$. As a consequence, the claim of the lemma follows.
\end{proof}

The next lemma establishes an $L^2$-bound for $\psi$, and at the same time improves the bound on the $L^2$-norm of $w$, $\nabla w$, and $\sqrt{\eta} \nabla \psi$.

\begin{lemma}
\label{lemmapsi}
Let Assumption \ref{assumplocality} be satisfied. Then, there exist positive constants $C = C(\rho, M, \omega_0) > 0$ and $\gamma = \gamma(\rho, M, \omega_0) > 0$ such that
\begin{align*}
&\int_{\mathbb R^3} (w^2 + |\nabla w|^2 + \psi^2 + \eta |\nabla \psi|^2) \xi^2 \, dx
\\&
\leq C \left( \int_{\{\eta < 1\}} (w^2 + \psi^2) \xi^2 \, dx + \int_{\mathbb R^3} (\delta m_c)^2 \xi^2 \, dx \right)
\end{align*}
where $\eta$ and $\xi$ (depending on $\gamma$) are defined in Assumption \ref{assumplocality}.
\end{lemma}

\begin{proof}
We rewrite \eqref{eqw} as
\[
-\Delta w + \frac53 \Big(u_1^\frac73 - u_2^\frac73\Big) - \frac{\phi_1 + \phi_2}{2} w = \frac{u_1 + u_2}{2} \psi
\]
and test with $\eta \psi \xi^2 \in H^1(\mathbb R^3)$. This gives
\begin{multline}
\label{eqpsi2}
\int_{\mathbb R^3} \frac{u_1 + u_2}{2} \eta \psi^2 \xi^2 \, dx \\
= \int_{\mathbb R^3} \nabla w \cdot \nabla (\eta \psi \xi^2) \, dx + \frac53 \int_{\mathbb R^3} \Big(u_1^\frac73 - u_2^\frac73\Big) \eta \psi \xi^2 \, dx - \int_{\mathbb R^3} \frac{\phi_1 + \phi_2}{2} \eta w \psi \xi^2 \, dx.
\end{multline}
Considering the first term on the right hand side of \eqref{eqpsi2}, we obtain
\begin{align*}
&\Big| \int_{\mathbb R^3} \nabla w \cdot \nabla (\eta \psi \xi^2) \, dx \Big| \\
&\leq \Big| \int_{\mathbb R^3} \nabla w \cdot \nabla \eta \; \psi \xi^2 \, dx \Big| + \Big| \int_{\mathbb R^3} \nabla w \cdot \nabla \psi \; \eta \xi^2 \, dx \Big| + 2 \Big| \int_{\mathbb R^3} \nabla w \cdot \nabla \xi \; \eta \psi \xi \, dx \Big|
\\
&\leq \Big( \int_{\mathbb R^3} |\nabla w|^2 \xi^2 \, dx \Big)^\frac12 \Big( \int_{\mathbb R^3} |\nabla \eta|^2 \psi^2 \xi^2 \, dx \Big)^\frac12
\\&\qquad
+ \Big( \int_{\mathbb R^3} |\nabla w|^2 \xi^2 \, dx \Big)^\frac12 \Big( \int_{\mathbb R^3} |\nabla \psi|^2 \eta^2 \xi^2 \, dx \Big)^\frac12
\\
&\qquad + 2 \Big( \int_{\mathbb R^3} |\nabla w|^2 \xi^2 \, dx \Big)^\frac12 \Big( \int_{\mathbb R^3} |\nabla \xi|^2 \eta^2 \psi^2 \, dx \Big)^\frac12.
\end{align*}
Taking $|\nabla \xi| \leq \xi$ and $|\nabla \eta|^2 \leq c_\eta(\rho) \eta$ into account, we obtain for any $\tau>0$ that there exists a constant $C(\tau) > 0$ such that
\begin{equation}
\label{eqpsi2term1}
\Big| \int_{\mathbb R^3} \nabla w \cdot \nabla (\eta \psi \xi^2) \, dx \Big| \leq C(\tau) \int_{\mathbb R^3} ( |\nabla w|^2 + \eta |\nabla \psi|^2 ) \xi^2 \, dx + \tau \int_{\mathbb R^3} \psi^2 \xi^2 \, dx.
\end{equation}
The next term in \eqref{eqpsi2} can be controlled for any $\tau>0$ using the $L^\infty$-bounds on $u_1$ and $u_2$ from Proposition \ref{propbounds} via
\begin{align}
&\frac53 \int_{\mathbb R^3} \Big(u_1^\frac73 - u_2^\frac73\Big) \eta \psi \xi^2 \, dx \label{eqpsi2term2} \\
&\qquad \qquad \leq \int_{\mathbb R^3} \eta |w| |\psi| \xi^2 \, dx \leq C(\tau) \int_{\mathbb R^3} w^2 \xi^2 \, dx + \tau \int_{\mathbb R^3} \psi^2 \xi^2 \, dx \nonumber
\end{align}
with some $C=C(\tau,M) > 0$. We proceed by estimating the last expression in \eqref{eqpsi2} using also Proposition~\ref{propbounds} as
\begin{align}
&\bigg|\int_{\mathbb R^3} \frac{\phi_1 + \phi_2}{2} \eta w \psi \xi^2 \, dx\bigg| 
\leq
\sum_{x \in \mathbb Z^3}
\bigg|\int_{B_1(x)} \frac{\phi_1 + \phi_2}{2} \eta w \psi \xi^2 \, dx
\bigg|
\label{eqpsi2term3}
\\
&\leq \sum_{x \in \mathbb Z^3} \Big\| \frac{\phi_1 + \phi_2}{2} \Big\|_{L^2_{\mathrm{uloc}}(\mathbb R^3)} \| w \xi \|_{L^4(B_1(x))} \| \eta \psi \xi \|_{L^4(B_1(x))} \nonumber \\
&\leq C \sum_{x \in \mathbb Z^3} \Big( \| w\xi \|_{L^2(B_1(x))} + \| \xi \nabla w \|_{L^2(B_1(x))} + \| w \nabla \xi \|_{L^2(B_1(x))} \Big) \nonumber \\
&\quad \Big( \| \eta \psi \xi \|_{L^2(B_1(x))} + \| \psi \xi \nabla \eta \|_{L^2(B_1(x))} + \| \eta \xi \nabla \psi \|_{L^2(B_1(x))} + \| \eta \psi \nabla \xi \|_{L^2(B_1(x))} \Big) \nonumber \\
&\leq \sum_{x \in \mathbb Z^3} \bigg[ C \Big( \| w\xi \|_{L^2(B_1(x))}^2 + \| \xi \nabla w \|_{L^2(B_1(x))}^2 + \| w \nabla \xi \|_{L^2(B_1(x))}^2 \Big) \nonumber \\
&~~ + \frac{\tau}{8} \Big( \| \eta \psi \xi \|_{L^2(B_1(x))}^2 + \| \psi \xi \nabla \eta \|_{L^2(B_1(x))}^2 + \| \eta \xi \nabla \psi \|_{L^2(B_1(x))}^2 + \| \eta \psi \nabla \xi \|_{L^2(B_1(x))}^2 \Big) \bigg] \nonumber \\
&\leq 8C \Big( \int_{\mathbb R^3} w^2 \xi^2 \, dx + \int_{\mathbb R^3} |\nabla w|^2 \xi^2 \, dx + \int_{\mathbb R^3} w^2 |\nabla \xi|^2 \, dx \Big) \nonumber \\
&\quad + \tau \Big( \int_{\mathbb R^3} \eta^2 \psi^2 \xi^2 \, dx + \int_{\mathbb R^3} |\nabla \eta|^2 \psi^2 \xi^2 \, dx + \int_{\mathbb R^3} \eta^2 |\nabla \psi|^2 \xi^2 \, dx + \int_{\mathbb R^3} \eta^2 \psi^2 |\nabla \xi|^2 \, dx \Big) \nonumber \\
&\leq 16C \int_{\mathbb R^3} (w^2 + |\nabla w|^2) \xi^2 \, dx + (2 + c_\eta) \tau \int_{\mathbb R^3} \psi^2 \xi^2 \, dx + \tau \int_{\mathbb R^3} \eta |\nabla \psi|^2 \xi^2 \, dx \nonumber
\end{align}
where $\tau > 0 $ will be chosen sufficiently small. Returning to $\int_{\mathbb R^3} \psi^2 \xi^2 \, dx$, we make use of the lower bounds $\inf_{\mathbb R^3} u_i > 0$, $i \in \{1,2\}$, from Theorem \ref{theoreminf} and rewrite
\begin{align*}
\int_{\mathbb R^3} \psi^2 \xi^2 \, dx &\leq C \int_{\mathbb R^3} \frac{u_1 + u_2}{2} \psi^2 \xi^2 \, dx \\&
\leq C \int_{\{\eta < 1\}} \frac{u_1 + u_2}{2} \psi^2 \xi^2 \, dx + C \int_{\mathbb R^3} \frac{u_1 + u_2}{2} \eta \psi^2 \xi^2 \, dx
\end{align*}
with constants $C(\rho, M, \omega_0) > 0$. We now combine 
\eqref{eqpsi2}--\eqref{eqpsi2term3}, and find
\begin{align*}
&\int_{\mathbb R^3} \psi^2 \xi^2 \, dx
\\&
\leq C \Big( \int_{\{\eta < 1\}} \psi^2 \xi^2 \, dx + \int_{\mathbb R^3} (w^2 + |\nabla w|^2 + \eta |\nabla \psi|^2 ) \xi^2 \, dx \Big) + C \tau \int_{\mathbb R^3} \psi^2 \xi^2 \, dx
\end{align*}
where $\tau>0$ can be chosen arbitrarily small. Thanks to Lemma \ref{lemmaw}, we arrive at
\begin{align*}
&\int_{\mathbb R^3} (w^2 + |\nabla w|^2 + \psi^2 + \eta |\nabla \psi|^2) \xi^2 \, dx \\&
\leq C \Big( \int_{\mathbb R^3} (w^2 + \psi^2) |\nabla \xi|^2 \, dx + \int_{\{\eta < 1\}} (w^2 + \psi^2) \xi^2 \, dx + \int_{\mathbb R^3} (\delta m_c)^2 \xi^2 \, dx \Big)
\\&~~~~~
+ C \tau \int_{\mathbb R^3} \psi^2 \xi^2 \, dx
\end{align*}
with $\tau>0$ arbitrary. If we set both constants $\tau$ (from above) and $\gamma$ (from the definition of $\xi \coleq e^{-\gamma | \cdot - \, y |}$) to sufficiently small values, we may absorb $\int_{\mathbb R^3} \psi^2 \xi^2 \, dx$ and $\int_{\mathbb R^3} (w^2 + \psi^2) |\nabla \xi|^2 \, dx$ on the left hand side due to $|\nabla \xi| \leq \gamma \xi$, which entails the desired estimate.
\end{proof}

We are now in position to prove our locality result  Theorem~\ref{theorempartialijpsi}.
\begin{proof}[Proof of Theorem~\ref{theorempartialijpsi}]
In view of Lemma~\ref{lemmapsi}, for proving Theorem \ref{theorempartialijpsi} it suffices to show that the bound on the $L^2$-norm of $w$, $\nabla w$, $\psi$, and $\nabla \psi$ from Lemma \ref{lemmapsi} also serves as an upper bound for the $L^2$-norm of the second order partial derivatives $\partial_{ij} \psi$.

We first establish a bound for $\int_{\mathbb R^3} \eta |\Delta \psi|^2 \xi^2 \, dx$. From \eqref{eqpsi}, we derive
\begin{align*}
\int_{\mathbb R^3} \eta |\Delta \psi|^2 \xi^2 \, dx = 4 \pi \int_{\mathbb R^3} \eta (u_1 + u_2) w \Delta \psi \; \xi^2 \, dx - 4 \pi \int_{\mathbb R^3} \eta\, \delta m_c\, \Delta \psi \; \xi^2 \, dx.
\end{align*}
Using Young's inequality and absorption as well as the bounds from Proposition~\ref{propbounds} and Lemma \ref{lemmapsi},
we arrive at
\begin{equation}
\label{eqlaplacepsi}
\int_{\mathbb R^3} \eta |\Delta \psi|^2 \xi^2 \, dx \leq C \Big( \int_{\{\eta < 1\}} (w^2 + \psi^2) \xi^2 \, dx + \int_{\mathbb R^3} (\delta m_c)^2 \xi^2 \, dx \Big).
\end{equation}
We will now employ integration by parts to arrive at
\begin{align*}
&\sum_{i,j} \int_{\mathbb R^3} \eta^2 |\partial_{ij} \psi|^2 \xi^2 \, dx \\
&= \sum_{i,j} \Big( -2 \int_{\mathbb R^3} \eta \partial_j \eta \partial_i \psi \partial_{ij} \psi \xi^2 \, dx - \int_{\mathbb R^3} \eta^2 \partial_i \psi \partial_{ijj} \psi \xi^2 \, dx - 2 \int_{\mathbb R^3} \eta^2 \partial_i \psi \partial_{ij} \psi \xi \partial_j \xi \, dx \Big) \\
&= \sum_{i,j} \Big( -2 \int_{\mathbb R^3} \eta \partial_j \eta \partial_i \psi \partial_{ij} \psi \xi^2 \, dx + 2 \int_{\mathbb R^3} \eta \partial_i \eta \partial_i \psi \partial_{jj} \psi \xi^2 \, dx \\
&\qquad + \int_{\mathbb R^3} \eta^2 \partial_{ii} \psi \partial_{jj} \psi \xi^2 \, dx + 2 \int_{\mathbb R^3} \eta^2 \partial_i \psi \partial_{jj} \psi \xi \partial_i \xi \, dx - 2 \int_{\mathbb R^3} \eta^2 \partial_i \psi \partial_{ij} \psi \xi \partial_j \xi \, dx \Big).
\end{align*}
We utilize the bounds $|\nabla \xi| \leq \xi$ and $|\nabla \eta|^2 \leq c_\eta(\rho) \eta$ and continue as
\begin{align*}
&\sum_{ij} \int_{\mathbb R^3} \eta^2 |\partial_{ij} \psi|^2 \xi^2 \, dx \\
&\quad \leq C \| \eta^\frac12 \xi \nabla \psi \|_{L^2(\mathbb R^3)} \Big( \sum_{ij} \| \eta \xi \partial_{ij} \psi \|_{L^2(\mathbb R^3)} + \| \eta \xi \Delta \psi \|_{L^2(\mathbb R^3)} \Big) + \int_{\mathbb R^3} \eta |\Delta \psi|^2 \xi^2 \, dx \\
&\quad \leq C \Big( \int_{\mathbb R^3} \eta |\nabla \psi|^2 \xi^2 \, dx + \int_{\mathbb R^3} \eta |\Delta \psi|^2 \xi^2 \, dx \Big) + \frac12 \sum_{ij} \int_{\mathbb R^3} \eta^2 |\partial_{ij} \psi|^2 \xi^2 \, dx.
\end{align*}
The assertion of the lemma now follows from \eqref{eqlaplacepsi} and Lemma \ref{lemmapsi}.
\end{proof}

We finally establish the locality result for the TFW energy.

\begin{proof}[Proof of Corollary~\ref{prop2diff}]
The difference of the TFW energy for $m_1$ and $m_2$ is given by
\begin{align}
\nonumber
E_{Q_1}[m_1]-E_{Q_1}[m_2]
=&\int_{Q_1} |\nabla u_1|^2 -  |\nabla u_2|^2 \,dx
\\&
\nonumber
+ \int_{Q_1} u_1^\frac{10}3 - u_2^\frac{10}3 \,dx
\\&
\label{eqtfwenergydifference}
+ \frac12 \int_{Q_1} (m_{c,1} - u_1^2) \phi_1 - (m_{c,2} - u_2^2) \phi_2 \, dx
\\&\nonumber
+ \sum_{x \in \mathbb P_1 \cap Q_1} c_{1,x} (\phi_1 - \phi_{1,x})(x)
- \sum_{x \in \mathbb P_2 \cap Q_1} c_{2,x} (\phi_2 - \phi_{2,x})(x).
\end{align}
We recall $\xi := e^{-\gamma | \cdot - \, y |}$, where $y := \dashint_{Q_1} x \, dx$ is the center of $Q_1$. We further recall the definition of $\eta$ from Notation \ref{noteta} and find using also Proposition~\ref{propbounds} and Lemma~\ref{lemmapsi}
\begin{align*}
&\Big| \int_{Q_1} ( |\nabla u_1|^2 - |\nabla u_2|^2 ) \, dx \Big| \leq C \int_{\mathbb R^3} \big| |\nabla u_1|^2 - |\nabla u_2|^2 \big| \xi^2 \, dx
\\
&\leq C \| (\nabla u_1 + \nabla u_2) \xi \|_{L^2(\mathbb R^3)} \| (\nabla u_1 - \nabla u_2) \xi \|_{L^2(\mathbb R^3)}
\\&
\leq C \Big( \int_{\{\eta < 1\}} (w^2 + \psi^2) \xi^2 \, dx + \int_{\supp(\delta m_c)} (\delta m_c)^2 \xi^2 \, dx \Big)^\frac 1 2
\\
&\leq C \Big( e^{-\gamma \, \dist(\{\eta < 1\}, Q_1)} \int_{\{\eta < 1\}} (w^2 + \psi^2) \xi \, dx
\\&\qquad\quad
+ e^{-\gamma \, \dist(\supp(\delta m_c), Q_1)} \int_{\supp(\delta m_c)} (\delta m_c)^2 \xi \, dx \Big)^\frac12 \\
&\leq C \Big( e^{2 \rho \gamma} e^{-\gamma \, \dist(\mathbb P', Q_1)} \int_{\{\eta < 1\}} (w^2 + \psi^2) \xi \, dx
\\&\qquad\quad
+ e^{-\gamma \, \dist(\supp(\delta m_c), Q_1)} \int_{\supp(\delta m_c)} (\delta m_c)^2 \xi \, dx \Big)^\frac12.
\end{align*}
Consequently, we have
\begin{align*}
&
\Big| \int_{Q_1} ( |\nabla u_1|^2 - |\nabla u_2|^2 ) \, dx \Big|
\\&
\leq C e^{-c \, \dist(\supp(m_1 - m_2), \, Q_1)} \Big( \int_{\mathbb R^3} (w^2 + \psi^2) \xi \, dx + \int_{\mathbb R^3} (\delta m_c)^2 \xi \, dx \Big)^\frac12,
\end{align*}
which (using also the bounds from Proposition~\ref{propbounds}) implies the desired estimate for the first term in \eqref{eqtfwenergydifference}. Concerning the second expression, we derive (using again Proposition~\ref{propbounds})
\[
\Big| \int_{Q_1} \big( u_1^\frac{10}3 - u_2^\frac{10}{3} \big) \, dx \Big| = \Big| \int_{Q_1} \frac{10}{3} v(x)^\frac73 \big( u_1(x) - u_2(x) \big) \, dx \Big| \leq C \bigg( \int_{Q_1} (u_1 - u_2)^2 \, dx \bigg)^\frac12
\]
where $v(x) \in [u_1(x), u_2(x)]$ for all $x \in Q_1$. As for the previous term, we obtain from Lemma \ref{lemmapsi}
\begin{multline*}
\Big| \int_{Q_1} \big( u_1^\frac{10}3 - u_2^\frac{10}{3} \big) \, dx \Big| \leq C \Big( \int_{\mathbb R^3} (u_1 - u_2)^2 \xi^2 \, dx \Big)^\frac12 \\
\leq C \Big( \int_{\{\eta < 1\}} (w^2 + \psi^2) \xi^2 \, dx + \int_{\supp(\delta m_c)} (\delta m_c)^2 \xi^2 \, dx \Big)^\frac12 \leq C e^{-c \, \dist(\supp(m_1 - m_2), \, Q_1)}.
\end{multline*}
The first part of the Coulomb energy in \eqref{eqtfwenergydifference} can be estimated via
\begin{multline*}
\Big| \int_{Q_1} \big( (m_{c,1} - u_1^2) \phi_1 - (m_{c,2} - u_2^2) \phi_2 \big) \, dx \Big| \\
\leq \int_{Q_1} \big( m_{c,1} |\phi_1 - \phi_2| + |\phi_2| |m_{c,1} - m_{c,2}| + u_1^2 |\phi_1 - \phi_2| + |\phi_2| |u_1^2 - u_2^2| \big) \, dx.
\end{multline*}
Due to $m_{c,i}, \phi_i \in L^2_\mathrm{uloc}(\mathbb R^3)$ and $u_i \in L^\infty(\mathbb R^3)$ (see Proposition~\ref{propbounds}), one concludes that
\begin{multline*}
\Big| \int_{Q_1} \big( (m_{c,1} - u_1^2) \phi_1 - (m_{c,2} - u_2^2) \phi_2 \big) \, dx \Big| \\
\leq C \Big( \| (\phi_1 - \phi_2) \xi \|_{L^2(\mathbb R^3)} + \| (m_{c,1} - m_{c,2}) \xi \|_{L^2(\mathbb R^3)} + \| (u_1 - u_2) \xi \|_{L^2(\mathbb R^3)} \Big)
\end{multline*}
and, hence, arrives at the desired bound by applying Lemma \ref{lemmapsi} and the same arguments as above. We are left to control the change of the Coulomb energy related to the atomic nuclei in \eqref{eqtfwenergydifference}, which we may bound by
\begin{multline*}
\Big| \sum_{x \in \mathbb P_1 \cap Q_1} c_{1,x} (\phi_1 - \phi_{1,x})(x) - \sum_{x \in \mathbb P_2 \cap Q_1} c_{2,x} (\phi_2 - \phi_{2,x})(x) \Big| \\
\leq \sum_{x \in (\mathbb P_1\cup \mathbb P_2) \cap Q_1} \big( |c_{1,x}| |(\phi_1 - \phi_{1,x} - \phi_2 + \phi_{2,x})(x)| + |(\phi_2 - \phi_{2,x})(x)| |c_{1,x} - c_{2,x}| \big).
\end{multline*}
In the case that $\dist(\supp(m_1 - m_2), \, Q_1) \leq 2\rho$, we observe that \eqref{eq2diff} holds true for the right hand side of the previous equation as it is bounded by a constant (due to the uniform bound on $\phi_i - \phi_{i,x}$) and as $e^{-c \, \dist(\supp(m_1 - m_2), \, Q_1)}$ is bounded from below by a positive constant. And if $\dist(\supp(m_1 - m_2), \, Q_1) > 2\rho$, we know that $\phi_1 - \phi_2 \in H^2(Q_1) \hookrightarrow C^{0,\frac12}(Q_1)$ and $\phi_{1,x} = \phi_{2,x}$ as well as $c_{1,x}=c_{2,x}$ for all $x \in (\mathbb P_1\cup \mathbb P_2) \cap Q_1$. As a consequence, in this case we have
\begin{align*}
&\Big| \sum_{x \in (\mathbb P_1 \cup \mathbb P_2) \cap Q_1} c_{1,x} (\phi_1 - \phi_{1,x})(x) - \sum_{x \in (\mathbb P_1 \cup \mathbb P_2) \cap Q_1} c_{2,x} (\phi_2 - \phi_{2,x})(x) \Big|
\\&
\leq C \sum_{x \in (\mathbb P_1 \cup \mathbb P_2) \cap Q_1} |c_{1,x}| |(\phi_1 - \phi_2)(x)| \leq C \| \phi_1 - \phi_2 \|_{H^2(Q_1)}
\\&
\leq C \bigg( \int_{Q_1} \Big( |\phi_1 - \phi_2|^2 + \eta |\nabla (\phi_1 - \phi_2)|^2 + \eta^2 \sum_{ij} |\partial_{ij} (\phi_1 - \phi_2)|^2 \Big) \, dx \bigg)^\frac12
\\&
\leq C \bigg( \int_{\mathbb R^3} \Big( |\phi_1 - \phi_2|^2 + \eta |\nabla (\phi_1 - \phi_2)|^2 + \eta^2 \sum_{ij} |\partial_{ij} (\phi_1 - \phi_2)|^2 \Big) \xi^2 \, dx \bigg)^\frac12
\end{align*}
and we may now employ Theorem \ref{theorempartialijpsi} and the bound from Proposition~\ref{propbounds} to finish the proof.
\end{proof}

\section{Uniform Bounds on Solutions to the TFW Equations}
\label{secsolution}

For our arguments we need uniform estimates on the solutions to the TFW equations which depend only on the parameters $\rho$, $M$, and $\omega_0$. For this reason, we repeat some of the calculations of \cite{CLBL98,NO17} to show that they do in fact yield uniform estimates.
The following lemma and its proof are based on similar considerations in \cite{CLBL98,NO17}. 
\begin{lemma}
\label{lemmapositiveoperator}
Let $a \in L^2_\mathrm{uloc}(\mathbb R^3)$ and suppose there exists some $u \in H^2_\mathrm{uloc}(\mathbb R^3)$ with $\inf_{B_R(0)} u > 0$ for all $R>0$ and $(-\Delta + a) u = 0$. Then, 
\[
\langle w, (-\Delta + a) w \rangle \coleq \int_{\mathbb R^3} (|\nabla w|^2 + aw^2) \, dx \geq 0
\]
for all $w \in H^1(\mathbb R^3)$.
\end{lemma}

\begin{proof}
We first note that $u \in L^\infty(\mathbb R^3) \cap C(\mathbb R^3)$ due to the uniform boundedness of $\|u\|_{H^2(B_1(x))}$ for every $x \in \mathbb R^3$. Regularizing $a \in L^2_\mathrm{uloc}(\mathbb R^3)$ (e.g. by convolution with some mollifier), we obtain a family of functions $a_\varepsilon \in C^\infty(\mathbb R^3)$ which converge for $\varepsilon \rightarrow 0$ to $a$ in $L^2(B_R)$ for any $R>0$. A standard result on differential operators \cite{GT01} ensures that the minimal eigenvalue $\lambda_\varepsilon$ of $-\Delta + a_\varepsilon$ as an operator on $B_R$ with Dirichlet boundary conditions is simple and is associated with a non-negative eigenfunction $v_\varepsilon \in H_0^1(B_R)$, $\|v_\varepsilon\|_{L^2(B_R)} = 1$. From the equation 
\begin{align*}
(-\Delta + a_\varepsilon) v_\varepsilon = \lambda_\varepsilon v_\varepsilon
\end{align*}
and elliptic regularity theory, we deduce $v_\varepsilon \in H^3(B_R) \hookrightarrow C^{1,\frac12}(\ol{B_R})$, hence $\nabla v_\varepsilon$ is well-defined and $\frac{\partial v_\varepsilon}{\partial n} \leq 0$ on $\partial B_R$. Moreover, it holds that
\begin{align*}
\lambda_\varepsilon = \int_{B_R} (|\nabla v_\varepsilon|^2 + a_\varepsilon v_\varepsilon^2) \, dx = \inf_{\genfrac{}{}{0pt}{2}{v \in H_0^1(B_R)}{\|v\|_{L^2(B_R)} = 1}} \int_{B_R} (|\nabla v|^2 + a_\varepsilon v^2) \, dx.
\end{align*}
We shall now prove that the eigenvalues $\lambda_\varepsilon$ are bounded for $\varepsilon \rightarrow 0$. For any fixed $v_\ast \in H^1_0(B_R)$ with $\|v_\ast\|_{L^2(B_R)} = 1$, we have $\lambda_\varepsilon \leq \int_{B_R} (|\nabla v_\ast|^2 + a_\varepsilon v_\ast^2) \, dx \leq \|v_\ast\|_{H^1(B_R)}^2 + C \|a_\varepsilon\|_{L^2(B_R)} \|v_\ast\|_{H^1(B_R)}^2$ where the last expression is bounded due to $a_\varepsilon \rightarrow a$ in $L^2(B_R)$. This yields an upper bound of the form $\lambda_\varepsilon \leq C$.
In addition, the equation $\int_{B_R} |\nabla v_\varepsilon|^2 \, dx = \lambda_\varepsilon - \int_{B_R} a_\varepsilon v_\varepsilon^2 \, dx$ and the Gagliardo--Nirenberg--Sobolev type estimate
\begin{align*}
\|v_\varepsilon\|_{L^4(B_R)} \leq C \|v_\varepsilon\|_{L^2(B_R)}^\frac14 \|v_\varepsilon\|_{H^1(B_R)}^\frac34 = C \|v_\varepsilon\|_{H^1(B_R)}^\frac34
\end{align*}
further implies
\begin{align*}
\Big| \int_{B_R} a_\varepsilon v_\varepsilon^2 \, dx \Big|
&\leq C \|a_\varepsilon\|_{L^2(B_R)} \|v_\varepsilon\|_{H^1(B_R)}^\frac32 \leq C \|a_\varepsilon\|_{L^2(B_R)}^4 + \frac12 \|v_\varepsilon\|_{H^1(B_R)}^2
\\&
\leq C + \frac12 \|\nabla v_\varepsilon\|_{L^2(B_R)}^2
\end{align*}
where we utilized the normalization of $v_\varepsilon$ and the boundedness of $\|a_\varepsilon\|_{L^2(B_R)}$. This results in $\| \nabla v_\varepsilon \|_{L^2(B_R)}^2 \leq \lambda_\varepsilon + C + \frac12 \|\nabla v_\varepsilon\|_{L^2(B_R)}^2$, which provides both a lower bound for $\lambda_\varepsilon$ of the form $\lambda_\varepsilon\geq -C$ and an upper bound for $\| v_\varepsilon \|_{H^1(B_R)}$. Up to a subsequence, we thus know that $\lambda_\varepsilon$ 
converges 
for $\varepsilon \rightarrow 0$.

In fact, one can show that $\lim_{\varepsilon \rightarrow 0} \lambda_\varepsilon \geq 0$ holds true. To see this, we calculate 
\begin{align*}
&\lambda_\varepsilon \int_{B_R} u v_\varepsilon \, dx = \int_{B_R} u (-\Delta + a_\varepsilon) v_\varepsilon \, dx
\\&
= - \int_{\partial B_R} u \frac{\partial v_\varepsilon}{\partial n} \, dS + \int_{B_R} \nabla u \cdot \nabla v_\varepsilon \, dx + \int_{B_R} a_\varepsilon u v_\varepsilon \, dx
\\&
= - \int_{\partial B_R} u \frac{\partial v_\varepsilon}{\partial n} \, dS + \int_{B_R} (-\Delta + a) u v_\varepsilon \, dx + \int_{B_R} (a_\varepsilon - a) u v_\varepsilon \, dx
\\&
\geq -\|a_\varepsilon - a\|_{L^2(B_R)} \|u\|_{L^\infty(B_R)}
\end{align*}
where we have employed $u \geq 0$ and $\frac{\partial v_\varepsilon}{\partial n} \leq 0$ on $\partial B_R$ as well as $(-\Delta + a)u = 0$. By arguing that $1 = \int_{B_R} v_\varepsilon^2 \, dx \leq (\int_{B_R} v_\varepsilon \, dx)^\frac12 (\int_{B_R} v_\varepsilon^3 \, dx)^\frac12 \leq C \|v_\varepsilon\|_{L^1(B_R)}^\frac12 \|v_\varepsilon\|_{H^1(B_R)}^\frac32 \leq C \|v_\varepsilon\|_{L^1(B_R)}^\frac12$, we conclude that 
\[
\int_{B_R} u v_\varepsilon \, dx \geq \inf_{B_R} u \int_{B_R} v_\varepsilon \, dx \geq c > 0.
\]
Consequently, $\lambda_\varepsilon \geq - (\int_{B_R} u v_\varepsilon \, dx)^{-1} \|a_\varepsilon - a\|_{L^2(B_R)} \|u\|_{L^\infty(B_R)} \rightarrow 0$ for $\varepsilon \rightarrow 0$.

Now choose some arbitrary $w \in H^1(\mathbb R^3)$ and a sequence $w_k \in C_c^\infty(\mathbb R^3)$, $w_k \rightarrow w$ in $H^1(\mathbb R^3)$. As
\[
\langle w, (-\Delta + a) w \rangle = \int_{\mathbb R^3} (|\nabla w|^2 + a w^2) \, dx = \lim_{k \rightarrow \infty} \int_{\mathbb R^3} (|\nabla w_k|^2 + a w_k^2) \, dx, 
\]
it suffices to verify that $\int_{\mathbb R^3} (|\nabla w_k|^2 + a w_k^2) \, dx \geq 0$ for all $k \in \mathbb N$. For fixed $k \in \mathbb N$, there exists some $R>0$ such that $\mathrm{supp} \; w_k \subset B_R(0)$, hence, $\int_{\mathbb R^3} (|\nabla w_k|^2 + a_\varepsilon w_k^2) \, dx \geq \lambda_\varepsilon \|w_k\|_{L^2(B_R(0))}^2$ for all $\varepsilon > 0$ and 
\[
\int_{\mathbb R^3} (|\nabla w_k|^2 + a w_k^2) \, dx = \lim_{\varepsilon \rightarrow 0} \int_{\mathbb R^3} (|\nabla w_k|^2 + a_\varepsilon w_k^2) \, dx \geq \|w_k\|_{L^2(B_R(0))}^2 \lim_{\varepsilon \rightarrow 0} \lambda_\varepsilon \geq 0.
\]
Finally, $\langle w, (-\Delta + a) w \rangle \geq 0$ is proven.
\end{proof}

Appropriate bounds on the solutions to the Thomas--Fermi--von Weizs\"acker equations \eqref{eqtfi} can be constructed with the help of Proposition \ref{proprn}. The proof mainly relies on arguments from \cite{CLBL98} and \cite{NO17} where corresponding estimates have been deduced in similar situations. 

\begin{proposition}
\label{proprn}
Let $M>0$ and $m = m_c + \sum_{x \in \mathbb P} c_x \delta_x$ where $m_c \in L^2_\mathrm{uloc}(\mathbb R^3)$, $m_c \geq 0$, $c_{x} > 0$ and $\mathbb P \subset \mathbb R^3$ such that $|x-y| \geq 4 \ol \rho > 0$ for all $x,y \in \mathbb P$ with $x \neq y$, and
\begin{equation}
\label{eqboundm}
\| m_c \|_{L^2_\mathrm{uloc}(\mathbb R^3)} + \sup_{x \in \mathbb R^3} \bigg( \sum_{y \in \mathbb P \cap B_1(x)} c_y^2 \bigg)^\frac12 \leq M.
\end{equation}
Then, there exists some $R_0 \geq 1$ such that for each $R_n \geq R_0$ and $m_{R_n} \coleq m \, \chi_{B_{R_n}(0)}$, there exists a solution $(u_{R_n}, \phi_{R_n}) \in H^1(\mathbb R^3) \times L^{2}_{\mathrm{uloc}}(\mathbb R^3)$, $u \geq 0$, which satisfies 
\begin{equation}
\label{eqtfrn}
\begin{cases}
-\Delta u_{R_n} + \frac53 u_{R_n}^\frac73 - \phi_{R_n} u_{R_n} = 0, \\
-\Delta \phi_{R_n} = 4\pi (m_{R_n} - u_{R_n}^2),
\end{cases}
\end{equation}
in the sense of distributions. Using the notation $\eta_\rho$ introduced in Notation \ref{noteta}, this solution satisfies the bounds
\begin{equation}
\label{eqbounds}
\begin{split}
\| u_{R_n} \|_{H^2_\mathrm{uloc}(\mathbb R^3)} &\leq C (1 + M^4), \\
\| u_{R_n} \|_{L^p_\mathrm{uloc}(\mathbb R^3)} &\leq C_p (1 + M^\frac34) \mbox{\ for all\ } 1 \leq p < 4, \\
\| \phi_{R_n} \|_{L^p_\mathrm{uloc}(\mathbb R^3)} &\leq C_p (1 + M) \mbox{\ for all\ } 1 \leq p < 3, \\
\| \phi_{R_n} \|_{W^{1,p}_\mathrm{uloc}(\mathbb R^3)} &\leq C_p \big(1 + M^\frac74 \big) \mbox{\ for all\ } 1 \leq p < \frac32, \\
\| \eta_\rho \partial_i \phi_{R_n} \|_{L^2_\mathrm{uloc}(\mathbb R^3)} &\leq C_\rho \big( 1 + M^\frac74 \big) \mbox{\ for all\ } 0 < \rho < \ol \rho, \ 1 \leq i \leq 3, \\
\| \eta_\rho \partial_{ij} \phi_{R_n} \|_{L^2_\mathrm{uloc}(\mathbb R^3)} &\leq C_\rho \big( 1 + M^\frac52 \big) \mbox{\ for all\ } 0 < \rho < \ol \rho, \ 1 \leq i,j \leq 3
\end{split}
\end{equation}
where $C, \, C_p, \, C_\rho >0$ are independent of $M$ and $R_n$.
\end{proposition}

\begin{proof}
We assume from now on that $m$ is not identically zero on $\mathbb R^3$. Consequently, there exists some $R_0 \geq 1$ such that $\int_{\mathbb R^3} m_{R_n} \, dx > 0$ holds true for all $R_n \geq R_0$. According to \cite[Thm. 7.19, Thm. 7.7, Thm. 7.8]{L81}, there exists a non-negative $u_{R_n} \in H^1(\mathbb R^3)$ satisfying $\int_{\mathbb R^3} u_{R_n}^2 \, dx = \int_{\mathbb R^3} m_{R_n} \, dx$ which is a solution to 
\[
-\Delta u_{R_n} + \frac53 u_{R_n}^\frac73 - \Big( (m_{R_n} - u_{R_n}^2) \ast \frac{1}{|\cdot|} \Big) u_{R_n} = -\theta_{R_n} u_{R_n}
\]
where $\theta_{R_n} \geq 0$ is the Lagrange multiplier associated to the charge constraint $\int_{\mathbb R^3} u_{R_n}^2 \, dx = \int_{\mathbb R^3} m_{R_n} \, dx$. By introducing
\begin{equation}
\label{eqphirndef}
\phi_{R_n} \coleq \big(m_{R_n} - u_{R_n}^2 \big) \ast \frac{1}{|\cdot|} - \theta_{R_n},
\end{equation}
we arrive at the Thomas--Fermi type equations
\begin{gather}
-\Delta u_{R_n} + \frac53 u_{R_n}^\frac73 - \phi_{R_n} u_{R_n} = 0, \label{equrn} \\
-\Delta \phi_{R_n} = 4\pi \big(m_{R_n} - u_{R_n}^2\big). \label{eqphirn}
\end{gather}

Due to \cite[Thm. 7.10, Thm. 7.13]{L81}, there even exists a solution $u_{R_n} \in H^2(\mathbb R^3) \hookrightarrow C^{0,\frac12}(\mathbb R^3)$, which satisfies $u(x) \rightarrow 0$ for $|x| \rightarrow \infty$ and $u_{R_n} > 0$ on $\mathbb R^3$. Moreover, we define $m_{c,R_n} \coleq m_c \, \chi_{B_{R_n}(0)}$ and derive for any $x \in \mathbb R^3$
\begin{align*}
&\Big\| (m_{R_n} - u_{R_n}^2) \ast \frac{1}{|\cdot|} \Big\|_{L^2(B_1(x))}
\\&
= \Big\| \sum_{y \in \mathbb P \cap B_{R_n}(0)} \frac{c_y}{|\cdot - y|} + (m_{c,R_n} - u_{R_n}^2) \ast \frac{1}{|\cdot|} \Big\|_{L^2(B_1(x))} \\
&\leq \sum_{y \in \mathbb P \cap B_{R_n}(0)} c_y \Big\| \frac{1}{|\cdot|} \Big\|_{L^2(B_1(0))} + \sqrt{\frac{4\pi}{3}} \Big\| (m_{c,R_n} - u_{R_n}^2) \ast \frac{1}{|\cdot|} \Big\|_{L^\infty(\mathbb R^3)} \\
&\leq C(R_{n}) \Big( M + \| m_{c,R_n} - u_{R_n}^2 \|_{L^\frac53(\mathbb R^3)} \Big\| \frac{\chi_{B_1(0)}}{|\cdot|} \Big\|_{L^\frac52(\mathbb R^3)}
\\&\qquad\qquad\quad
+ \| m_{c,R_n} - u_{R_n}^2 \|_{L^\frac75(\mathbb R^3)} \Big\| \frac{\chi_{\mathbb R^3 \backslash B_1(0)}}{|\cdot|} \Big\|_{L^\frac72(\mathbb R^3)} \Big)
\\
&\leq C(R_{n}) \big( M + \| m_{c,R_n} \|_{L^2(B_{R_n}(0))} + \| u_{R_n} \|_{L^\frac{10}{3}(\mathbb R^3)}^2 + \| u_{R_n} \|_{L^\frac{14}{5}(\mathbb R^3)}^2 \big) \\
&\leq C(R_n) \big( M + \| u_{R_n} \|_{H^1(\mathbb R^3)}^2 \big),
\end{align*}
where we also employed the Sobolev embedding. As a consequence, we get $\phi_{R_n} \in L^2_\mathrm{uloc}(\mathbb R^3)$. Analogously,
\begin{align*}
&\Big\| (m_{R_n} - u_{R_n}^2) \ast \frac{1}{|\cdot|^2} \Big\|_{L^\frac43(B_1(x))} \\&
= \Big\| \sum_{y \in \mathbb P \cap B_{R_n}(0)} \frac{c_y}{|\cdot - y|^2} + (m_{c,R_n} - u_{R_n}^2) \ast \frac{1}{|\cdot|^2} \Big\|_{L^\frac43(B_1(x))} \\
&\leq \sum_{y \in \mathbb P \cap B_{R_n}(0)} c_y \Big\| \frac{1}{|\cdot|^2} \Big\|_{L^\frac43(B_1(0))} + C \Big\| (m_{c,R_n} - u_{R_n}^2) \ast \frac{1}{|\cdot|^2} \Big\|_{L^3(\mathbb R^3)} \\
&\leq C(R_{n}) \Big( M + \| m_{c,R_n} - u_{R_n}^2 \|_{L^\frac53(\mathbb R^3)} \Big\| \frac{\chi_{B_1(0)}}{|\cdot|^2} \Big\|_{L^\frac{15}{11}(\mathbb R^3)}
\\&\qquad\qquad\quad
+ \| m_{c,R_n} - u_{R_n}^2 \|_{L^\frac75(\mathbb R^3)} \Big\| \frac{\chi_{\mathbb R^3 \backslash B_1(0)}}{|\cdot|^2} \Big\|_{L^\frac{21}{13}(\mathbb R^3)} \Big)
\\
&\leq C(R_n) \big( M + \| u_{R_n} \|_{H^1(\mathbb R^3)}^2 \big),
\end{align*}
which implies $\nabla \phi_{R_n} \in L^\frac43_\mathrm{uloc}(\mathbb R^3)$. Moreover, we will use the fact that for any $f \in L^p(\mathbb R^3)$, $g \in L^q(\mathbb R^3)$ and dual indices $p,\,q \in (1,\infty)$, the convolution $f \ast g$ is a continuous function tending to zero at infinity (see e.g. \cite[Lemma II.25]{LS77}). From the previous calculations, we thus know that $(m_{c,R_n} - u_{R_n}^2) \ast \frac{1}{|\cdot|} \in C(\mathbb R^3)$ and $\big((m_{c,R_n} - u_{R_n}^2) \ast \frac{1}{|\cdot|} \big)(x) \rightarrow 0$ for $|x| \rightarrow \infty$. As a result, $\phi_{R_n} \in C(\mathbb R^3 \backslash (\mathbb P \cap B_{R_n}(0)))$ and $\phi_{R_n}(x) \rightarrow -\theta_{R_n}$ for $|x| \rightarrow \infty$.

A pointwise lower bound (uniform in $R_n$) for $\phi_{R_n}$ can be obtained from the inequalities \cite[Prop. 8, Cor. 9]{Sol90}
\begin{equation}
\label{eqsolovej}
\begin{split}
\frac{9}{10} u_{R_n}^\frac43 &\leq (m_{R_n} - u_{R_n}^2) \ast \frac{1}{|\cdot|} + \Lambda, \\
0 &\leq \theta_{R_n} \leq \Lambda
\end{split}
\end{equation}
where $\Lambda > 0$ is a constant independent of $M$ and $R_n$. Thus, 
\begin{equation}
\label{eqphirnlowerbound}
\phi_{R_n} \geq -2\Lambda.
\end{equation}
A pointwise upper bound for $\phi_{R_n}$ cannot hold due to the point charges, but we may follow the arguments of \cite{CLBL98, NO17} to establish upper bounds for $\phi_{R_n}$ in $L^p_\mathrm{uloc}(\mathbb R^3)$, $p < 3$, which are uniform in $R_n$.

\medskip
\paragraph{\textbf{Step 1: $L^p$-bound on $\phi_{R_n}$}}
Let $\omega \in C_c^\infty(B_1(0))$ satisfying $0 \leq \omega \leq 1$, $\omega = 1$ on $B_\frac12(0)$ and $\int_{\mathbb R^3} \omega^2 \, dx = 1$. We further define $c_\omega \coleq \int_{\mathbb R^3} |\nabla \omega|^2 \, dx$ and $\omega_x \coleq \omega(\cdot - x)$. Applying Lemma \ref{lemmapositiveoperator}, we know that the operator $L_{R_n} \coleq -\Delta + \frac53 u_{R_n}^\frac43 - \phi_{R_n}$ is non-negative. Therefore,
\[
\langle \omega_x, L_{R_n} \omega_x \rangle = \int_{\mathbb R^3} |\nabla \omega_x|^2 \, dy + \int_{\mathbb R^3} \Big( \frac53 u_{R_n}^\frac43 - \phi_{R_n} \Big) \omega_x^2 \, dy \geq 0,
\]
and hence,
\[
\frac53 u_{R_n}^\frac43 \ast \omega^2 \geq \Big( \phi_{R_n} \ast \omega^2 - c_\omega \Big)_+.
\]
We now construct a (uniform in $R_n$) pointwise upper bound for the convolution $\phi_{R_n} \ast \omega^2$. First,
\[
-\Delta \big( \phi_{R_n} \ast \omega^2 \big) = 4 \pi \big( m_{R_n} \ast \omega^2 - u_{R_n}^2 \ast \omega^2 \big)
\]
and the first term on the right hand side can be estimated by
\[
( m_{R_n} \ast \omega^2 )(x) = \int_{B_1(x)} m_{R_n}(y) \omega^2(x-y) \, dy \leq \sum_{y \in \mathbb P \cap B_1(x)} c_y + \int_{B_1(x)} m_c(y) \, dy \leq C M
\]
for all $x \in \mathbb R^3$ with a constant $C > 0$ independent of $M$ and $R_n$. By employing Jensen's inequality, we control the second term via
\begin{align*}
4 \pi (u_{R_n}^2 \ast \omega^2)(x) &\geq \bigg( \frac53 \bigg)^\frac32 \int_{\mathbb R^3} u_{R_n}^2(x-y) \omega^2(y) \, dy \\
&\geq \bigg( \frac53 \bigg)^\frac32 \bigg( \int_{\mathbb R^3} u_{R_n}^\frac43(x-y) \omega^2(y) \, dy \bigg)^\frac32 \\
&= \Big( \frac53 u_{R_n}^\frac43 \ast \omega^2 \Big)^\frac32 \geq \big( \phi_{R_n} \ast \omega^2 - c_\omega \big)_+^\frac32.
\end{align*}
We thus have
\[
-\Delta \big( \phi_{R_n} \ast \omega^2 \big) + \big( \phi_{R_n} \ast \omega^2 - c_\omega \big)_+^\frac32 \leq C_\ast M,
\]
with a constant $C_\ast > 0$. Apart from that, one can easily show that $\phi_{R_n} \ast \omega^2$ is a continuous function (see e.g. \cite[Lemma II.25]{LS77}), which satisfies --- due to \eqref{eqphirndef} --- $(\phi_{R_n} \ast \omega^2)(x) \rightarrow -\theta_{R_n} \leq 0$ for $|x| \rightarrow \infty$.

We introduce the set 
\[
S \coleq \Big\{ x \in \mathbb R^3 \ \Big| \ \phi_{R_n} \ast \omega^2 - c_\omega > 0 \Big\},
\]
which is open and bounded due to the previous calculations. Furthermore, the constant and positive function $h \coleq (C_\ast M)^\frac23$ satisfies $-\Delta h + h_+^\frac32 = C_\ast M$ on $S$, which entails
\begin{align*}
-\Delta \big( \phi_{R_n} \ast \omega^2 - c_\omega \big) + \big( \phi_{R_n} \ast \omega^2 - c_\omega \big)_+^\frac32 &\leq -\Delta h + h_+^\frac32 \mbox{\quad on \ } S, \\
\phi_{R_n} \ast \omega^2 - c_\omega = 0 &\leq h \mbox{\quad on \ } \partial S.
\end{align*}
Thanks to the maximum principle, we arrive at $\phi_{R_n} \ast \omega^2 \leq c_\omega + C_\ast^\frac23 M^\frac23$ on $S$, but trivially also on $\mathbb R^3 \backslash S$. Therefore, 
\[
\phi_{R_n} \ast \omega^2 \leq C \big(1 + M^\frac23 \big)
\]
with a constant $C>0$ independent of $M$ and $R_n$.

In the case that $\phi_{R_n} \leq 0$ on $\mathbb R^3$, we have due to \eqref{eqphirnlowerbound} the pointwise bounds $-2\Lambda \leq \phi_{R_n} \leq 0$. Otherwise, the positive part $\phi_{R_n}^+$ is not identically zero, and we shall derive appropriate $L^p$-bounds for $\phi_{R_n}^+$. We first recall that $\phi_{R_n} \in C(\mathbb R^3 \backslash (\mathbb P \cap B_{R_n}(0)))$. In particular, $\phi_{R_n}^+$ is continuous away from the set $\mathbb P \cap B_{R_n}(0)$ and
\begin{equation*}
\phi_{R_n}^+ \ast \omega^2 = \phi_{R_n}^- \ast \omega^2 + \phi_{R_n} \ast \omega^2 \leq 2 \Lambda + C \big(1 + M^\frac23 \big) \leq C \big(1 + M^\frac23 \big)
\end{equation*}
with constants $C>0$ independent of $M$ and $R_n$. Now, choose some arbitrary $x_0 \in \mathbb R^3 \backslash (\mathbb P \cap B_{R_n}(0))$ satisfying $\phi_{R_n}(x_0) > 0$. On the one hand, we obtain the bound
\begin{align}
\int_{B_\frac12(x_0)} \phi_{R_n}^+(x) \, dx &\leq \int_{\mathbb R^3} \phi_{R_n}^+(x) \omega^2(x_0 - x) \, dx \label{eqphirn+omega2} \\
&= \big( \phi_{R_n}^+ \ast \omega^2 \big)(x_0) \leq C \big(1 + M^\frac23 \big). \nonumber
\end{align}
On the other hand, we may write
\[
\int_{B_\frac12(x_0)} \phi_{R_n}^+(x) \, dx = \int_0^\frac12 \int_{\partial B_\tau(x_0)} \phi_{R_n}^+(y) \, ds(y) \, d\tau,
\]
and we immediately see that there exists some $t \in (\frac14, \frac12)$ such that 
\begin{equation}
\label{eqboundaryintphirn}
\int_{\partial B_t(x_0)} \phi_{R_n}^+(y) \, ds(y) < 8 \int_{B_\frac12(x_0)} \phi_{R_n}^+(x) \, dx.
\end{equation}

Consider the boundary related problem
\begin{align*}
-\Delta \phi^{x_0}_1 &= 0 \mbox{\quad on \ } B_t(x_0), \\
\phi_1^{x_0} &= \phi_{R_n}^+ \mbox{\quad on \ } \partial B_t(x_0),
\end{align*}
as well as the two domain related problems 
\begin{align*}
-\Delta \phi^{x_0}_2 &= 4 \pi m_c \mbox{\quad on \ } B_t(x_0), \\
\phi_2^{x_0} &= 0 \mbox{\quad on \ } \partial B_t(x_0).
\end{align*}
and
\begin{align*}
-\Delta \phi^{x_0}_3 &= 4 \pi \sum_{y \in \mathbb P \cap B_{R_n}(0)} c_y \delta_y \mbox{\quad on \ } B_t(x_0), \\
\phi_3^{x_0} &= 0 \mbox{\quad on \ } \partial B_t(x_0).
\end{align*}
Because of
\[
-\Delta \phi_{R_n}^+ \leq  (-\Delta \phi_{R_n}) \chi_{\{\phi_{R_n} > 0\}} = 4 \pi \big( m_{R_n} - u_{R_n}^2 \big) \chi_{\{\phi_{R_n} > 0\}} \leq 4 \pi m_{R_n}
\]
we may employ the maximum principle to deduce $\phi_{R_n}^+ \leq \phi_1^{x_0} + \phi_2^{x_0} + \phi_3^{x_0}$ on $B_t(x_0)$. In particular, $\phi_{R_n}^+(x_0) \leq \phi_1^{x_0}(x_0) + \phi_2^{x_0}(x_0) + \phi_3^{x_0}(x_0)$ and we shall derive bounds for the three terms on the right hand side which are independent of $R_n$.

The first bound follows from the mean value property of harmonic functions and the estimates in \eqref{eqboundaryintphirn} and \eqref{eqphirn+omega2} via
\[
\phi_1^{x_0}(x_0) = \dashint_{\partial B_t(x_0)} \phi_{R_n}^+(y) \, ds(y) < \frac{8}{|\partial B_\frac14(0)|} \int_{B_\frac12(x_0)} \phi_{R_n}^+(x) \, dx \leq C (1 + M^\frac23),
\]
where the constant $C>0$ is independent of $M$ and $R_n$. For the second problem, we proceed as in \cite{NO17} and find a solution $\phi_2^{x_0} \in H^2(B_t(x_0)) \hookrightarrow C^{0,\frac12}(\ol{B_t(x_0)})$. This yields
\begin{multline*}
\phi_2^{x_0}(x_0) \leq \| \phi_2^{x_0} \|_{C^{0,\frac12}(\ol{B_t(x_0)})} \leq C \| \phi_2^{x_0} \|_{H^2(B_t(x_0))} \\
\leq C \| m_c \|_{L^2(B_t(x_0))} \leq C \| m_c \|_{L^2_{\mathrm{uloc}}(\mathbb R^3)} \leq C M
\end{multline*}
with $C>0$ independent of $M$ and $R_n$. The bound on $\phi_3^{x_0}(x_0)$ arises from a comparison of $\phi_3^{x_0}$ with 
\[
\wh{\phi}_3^{x_0} \coleq \sum_{y \in \mathbb P \cap B_{R_n}(0) \cap B_t(x_0)} \frac{c_y}{|\cdot - \, y|}.
\]
As $-\Delta \wh \phi_3^{x_0} = 4 \pi \sum_{y \in \mathbb P \cap B_{R_n}(0) \cap B_t(x_0)} c_y \delta_y = 4 \pi \sum_{y \in \mathbb P \cap B_{R_n}(0)} c_y \delta_y = -\Delta \phi_3^{x_0}$ in $B_t(x_0)$ and $\wh \phi_3^{x_0} \geq 0 = \phi_3^{x_0}$ on $\partial B_t(x_0)$, we have $\phi_3^{x_0} \leq \wh \phi_3^{x_0}$ in $B_t(x_0)$ and, hence,
\[
\phi_3^{x_0}(x_0) \leq \sum_{y \in \mathbb P \cap B_{R_n}(0) \cap B_t(x_0)} \frac{c_y}{|x_0 - y|} = \sum_{y \in \mathbb P \cap B_{R_n}(0)} \frac{c_y}{|x_0 - y|} \chi_{B_t(y)}(x_0).
\]
Together, we arrive at
\[
\phi_{R_n}^+(x_0) \leq C(1 + M) + \sum_{y \in \mathbb P \cap B_{R_n}(0)} \frac{c_y}{|x_0 - y|} \chi_{B_t(y)}(x_0),
\]
and as $x_0 \in \mathbb R^3 \backslash (\mathbb P \cap B_{R_n}(0))$ has been chosen arbitrarily, we further obtain 
\[
\phi_{R_n}^+(x) \leq C(1 + M) + \sum_{y \in \mathbb P \cap B_{R_n}(0)} \frac{c_y}{|x - y|} \chi_{B_t(y)}
\]
a.e. in $\mathbb R^3$ where $C>0$ is a constant independent of $M$ and $R_n$. For $p \in [1,3)$, we then conclude that 
\[
\| \phi_{R_n}^+ \|_{L^p_\mathrm{uloc}(\mathbb R^3)} 
\leq C_p (1 + M)
\]
where $C_p > 0$ denotes a constant depending only on $p$. Combining this estimate with the lower bound for $\phi_{R_n}$ in \eqref{eqphirnlowerbound}, entails --- as a first step --- the desired $L^p$-bound on $\phi_{R_n}$ in \eqref{eqbounds}.

\medskip
\paragraph{\textbf{Step 2: Further bounds}}
In order to establish the bounds on $u_{R_n}$ in \eqref{eqbounds}, we first utilize \eqref{eqsolovej} to find 
\begin{equation}
\label{equrnlp}
\| u_{R_n} \|_{L^p(B_1(x_0))} = \big\| u_{R_n}^\frac43 \big\|_{L^{\frac34 p}(B_1(x_0))}^\frac34 \leq C_p \Big(1 + \big\| \phi_{R_n} \big\|_{L^{\frac34 p}(B_1(x_0))}^\frac34 \Big) \leq C_p(1 + M^\frac34)
\end{equation}
for any $p \in [1,4)$ and $x_0 \in \mathbb R^3$. We recall the definition of the cut-off function $\wt \eta_\rho : [0,\infty) \rightarrow [0,1]$ from Notation \ref{noteta} and observe that $\eta_x : \mathbb R^3 \rightarrow [0,1]$, $\eta_x \coleq 1 - \wt \eta_1(|\cdot - x|)$ is another cut-off function satisfying $\eta_x = 1$ on $\ol{B_1(x)}$ and $\eta_x = 0$ on $\mathbb R^3 \backslash B_2(x)$. As an immediate consequence, there exists some $C>0$ such that $|\nabla \eta_x|^2 \leq C$ holds true on $\mathbb R^3$. Testing \eqref{equrn} with $\eta_{x_0}^2 u_{R_n}$ gives rise to
\begin{multline*}
\int_{B_2(x_0)} \eta_{x_0}^2 |\nabla u_{R_n}|^2 \, dx + 2 \int_{B_2(x_0)} \eta_{x_0} u_{R_n} \nabla u_{R_n} \cdot \nabla \eta_{x_0} \, dx \\
= -\frac53 \int_{B_2(x_0)} \eta_{x_0}^2 u_{R_n}^\frac{10}{3} \, dx + \int_{B_2(x_0)} \eta_{x_0}^2 \phi_{R_n} u_{R_n}^2 \, dx.
\end{multline*}
Applying Young's inequality and $|\nabla \eta_{x_0}| \leq C$ to the second integral on the left hand side, an absorption argument leads one to
\begin{align}
\int_{B_1(x_0)} |\nabla u_{R_n}|^2 \, dx &\leq 2 \int_{B_2(x_0)} \phi_{R_n} u_{R_n}^2 \, dx + C \int_{B_2(x_0)} u_{R_n}^2 \, dx \label{eqh1boundurn} \\
&\leq 2 \| \phi_{R_n} \|_{L^\frac52(B_2(x_0))} \| u_{R_n} \|^2_{L^\frac{10}{3}(B_2(x_0))} + C \| u_{R_n} \|^2_{L^2(B_2(x_0))} \nonumber \\
&\stackrel{\eqref{equrnlp}}{\leq} C(1 + M)(1 + M^\frac32) + C(1 + M^\frac32) \leq C(1 + M^\frac52). \nonumber
\end{align}
Now, consider the equation 
\[
-\Delta u_{R_n} = -\frac53 u_{R_n}^\frac73 + \phi_{R_n} u_{R_n}.
\]
As the right hand side belongs to $L^\frac74(B_2(x_0))$ (which will be detailed immediately), a standard result (see e.g. \cite[Theorem 8.17]{GT01}) ensures the following norm estimates on $B_2(x_0)$:
\begin{align*}
&\| u_{R_n} \|_{L^\infty(B_1(x_0))}
\\
&\leq C \bigg( \| u_{R_n} \|_{L^2(B_2(x_0))} + \Big\| -\frac53 u_{R_n}^\frac73 + \phi_{R_n} u_{R_n} \Big\|_{L^\frac74 (B_2(x_0))} \bigg)  \\
&\leq C \Big( \| u_{R_n} \|_{L^2(B_2(x_0))} + \| u_{R_n} \|_{L^\frac{49}{12}(B_2(x_0))}^\frac73 + \| \phi_{R_n} \|_{L^\frac{21}{8}(B_2(x_0))} \| u_{R_n} \|_{L^\frac{21}{4}(B_2(x_0))} \Big) \\
&\leq C \Big( \| u_{R_n} \|_{L^2(B_2(x_0))} + \| u_{R_n} \|_{H^1(B_2(x_0))}^\frac73 + \| \phi_{R_n} \|_{L^\frac{21}{8}(B_2(x_0))} \| u_{R_n} \|_{H^1(B_2(x_0))} \Big) \\
&\leq C \big( (1 + M^\frac34) + (1 + M^\frac{5}{4})^\frac73 + (1 + M)(1 + M^\frac54) \big) \\
&\leq C (1 + M^\frac{35}{12})
\end{align*}
where $C>0$ denotes various constants independent of $M$ and $R_n$. As a consequence, $\| u_{R_n} \|_{L^\infty(\mathbb R^3)} \leq C (1 + M^\frac{35}{12})$ and $\frac53 u_{R_n}^\frac73 - \phi_{R_n} u_{R_n} \in L^2_\mathrm{uloc}(\mathbb R^3)$. By applying standard elliptic regularity theory to \eqref{equrn} we thus conclude that 
\begin{align}
&\| u_{R_n} \|_{H^2(B_1(x_0))} \leq C \Big( \Big\| \frac53 u_{R_n}^\frac73 - \phi_{R_n} u_{R_n} \Big\|_{L^2(B_2(x_0))} + \| u_{R_n} \|_{H^1(B_2(x_0))} \Big) \nonumber \\
&\ \leq C \Big( \| u_{R_n} \|_{H^1(B_2(x_0))}^\frac73 + \| \phi_{R_n} \|_{L^2(B_2(x_0))} \| u_{R_n} \|_{L^\infty(\mathbb R^3)} + \| u_{R_n} \|_{H^1(B_2(x_0))} \Big) \nonumber \\
&\ \leq C \big( (1 + M^\frac{5}{4})^\frac73 + (1 + M)(1 + M^\frac{35}{12}) + (1 + M^\frac54) \big) \nonumber \\
&\ \leq C (1 + M^\frac{47}{12}). \nonumber
\end{align}

For establishing the remaining bounds on $\phi_{R_n}$, we split
\[
\phi_{R_n} = \sum_{y \in \mathbb P \cap B_{R_n}(0)} \frac{c_y}{|\cdot - \, y|} \omega(\cdot - y) + \phi_{R_n}^c
\]
where all singularities of $\phi_{R_n}$ are collected within the first term. The cut-off function $\omega \in C_c^\infty(\mathbb R^3)$ satisfying $\omega = 1$ on $B_1(0)$ and $\omega = 0$ on $\mathbb R^3 \backslash B_2(0)$ enforces each contribution from the Coulomb potential to have finite range. The non-singular function $\phi_{R_n}^c$ is then subject to
\begin{equation}
\label{eqlaplacephirnc}
-\Delta \phi_{R_n}^c = 4 \pi (m^c_{R_n} + u_{R_n}^2) + \sum_{y \in \mathbb P \cap B_{R_n}(0)} c_y \Big( -2 \frac{\cdot - y}{|\cdot - \, y|^3} \cdot \nabla \omega (\cdot - y) + \frac{1}{|\cdot - \, y|} \Delta \omega (\cdot - y) \Big).
\end{equation}
Testing this equation with $\omega(\cdot - y) \phi_{R_n}^c$ on $B_2(x_0)$ for some arbitrary $x_0 \in \mathbb R^3$ entails 
\begin{align*}
&\int_{B_1(x_0)} |\nabla \phi_{R_n}^c|^2 \, dx \leq \int_{B_2(x_0)} \omega(\cdot - x_0) |\nabla \phi_{R_n}^c|^2 \, dx \\
&= - \int_{B_2(x_0)} \Delta \phi_{R_n}^c \omega(\cdot - x_0) \phi_{R_n}^c \, dx - \int_{B_2(x_0)} \nabla \omega(\cdot - x_0) \cdot \phi_{R_n}^c \nabla \phi_{R_n}^c \, dx \\
&\leq C \int_{B_2(x_0)} \bigg( (|m^c_{R_n}| + u_{R_n}^2) |\phi_{R_n}^c| + |\nabla \omega(\cdot - x_0)| |\phi_{R_n}^c| |\nabla \phi_{R_n}^c| \\
&+ \sum_{y \in \mathbb P \cap B_{R_n}(0)} c_y \Big( \frac{1}{|\cdot - \, y|^2} |\nabla \omega (\cdot - y)| + \frac{1}{|\cdot - \, y|} |\Delta \omega (\cdot - y)| \Big) |\phi_{R_n}^c| \bigg) dx.
\end{align*}
As the expression inside the brackets in the last line is bounded by a constant which only depends on the choice of $\omega$, we deduce
\begin{align*}
&\int_{B_1(x_0)} |\nabla \phi_{R_n}^c|^2 \, dx \\
&\quad \leq C \Big( (\|m^c_{R_n}\|_{L^2_\mathrm{uloc}(\mathbb R^3)} + \|u_{R_n}\|_{H^1_\mathrm{uloc}(\mathbb R^3)}^2) \|\phi_{R_n}^c\|_{L^2_\mathrm{uloc}(\mathbb R^3)} \\
&\quad + \frac{1}{4\gamma} \|\phi_{R_n}^c\|_{L^2_\mathrm{uloc}(\mathbb R^3)}^2 + \gamma \|\nabla \phi_{R_n}^c\|_{L^2_\mathrm{uloc}(\mathbb R^3)}^2 + \sum_{y \in \mathbb P \cap B_{2}(x_0)} c_y \|\phi_{R_n}^c\|_{L^1_\mathrm{uloc}(\mathbb R^3)} \Big)
\end{align*}
where $\gamma > 0$ will be chosen subsequently. Besides, we observe that
\begin{align*}
\|\phi_{R_n}^c\|_{L^2_\mathrm{uloc}(\mathbb R^3)} &\leq \|\phi_{R_n}\|_{L^2_\mathrm{uloc}(\mathbb R^3)} + \sup_{x \in \mathbb R^3} \sum_{y \in \mathbb P \cap B_{2}(x)} c_y \Big\| \frac{\omega(\cdot - y)}{|\cdot - \, y|}\Big\|_{L^2(B_1(x))} \\
&\leq C(1 + M) + \sup_{x \in \mathbb R^3} \sum_{y \in \mathbb P \cap B_3(x)} c_y \, C \leq C(1 + M).
\end{align*}
For $\gamma > 0$ sufficiently small and together with \eqref{eqh1boundurn}, we arrive at
\begin{align*}
&\int_{B_1(x_0)} |\nabla \phi_{R_n}^c|^2 \, dx \\
&\quad \leq C \Big( (1 + M^\frac52)(1 + M) + (1 + M^2) + M(1 + M) \Big) + \frac12 \| \nabla \phi_{R_n}^c \|_{L^2_\mathrm{uloc}(\mathbb R^3)}^2,
\end{align*}
and, hence,
\begin{equation}
\label{eqgradphirnc}
\| \nabla \phi_{R_n}^c \|_{L^2_\mathrm{uloc}(\mathbb R^3)} \leq C \big( 1 + M^\frac74 \big).
\end{equation}
A similar estimate can be derived for the second order derivatives $\partial_{ij} \phi_{R_n}^c$, $1 \leq i,j \leq 3$. To this end, we first note that 
\begin{align*}
&-\int_{B_2(x_0)} \Delta \phi_{R_n}^c \omega \Delta \phi_{R_n}^c \, dx = \int_{B_2(x_0)} \nabla \phi_{R_n}^c \cdot \nabla \omega \, \Delta \phi_{R_n}^c \, dx \\
&\qquad \qquad - \sum_{i,j} \int_{B_2(x_0)} \partial_i \phi_{R_n}^c \, \partial_j \omega \, \partial_{ij} \phi_{R_n}^c \, dx - \sum_{i,j} \int_{B_2(x_0)} \omega |\partial_{ij} \phi_{R_n}^c|^2 \, dx.
\end{align*}
This enables one to estimate
\begin{align*}
&\sum_{i,j} \int_{B_1(x_0)} |\partial_{ij} \phi_{R_n}^c|^2 \, dx \\
&\qquad \leq C \int_{B_2(x_0)} \Big( |\nabla \phi_{R_n}^c|^2 + |\Delta \phi_{R_n}^c|^2 \Big) dx + \sum_{i,j} \int_{B_2(x_0)} |\partial_i \phi_{R_n}^c| \, |\partial_j \omega| \, |\partial_{ij} \phi_{R_n}^c| \, dx.
\end{align*}
From \eqref{eqlaplacephirnc} and by arguing as above, we know that
\[
\int_{B_2(x_0)} |\Delta \phi_{R_n}^c|^2 \, dx \leq C \Big( \int_{B_2(x_0)} \big( (m_{R_n}^c)^2 + u_{R_n}^4 \big) \, dx + \sum_{y \in \mathbb P \cap B_{3}(x_0)} c_y^2 \Big)
\]
where the constant $C > 0$ only depends on the choice of $\omega$. 
Consequently,
\begin{align*}
&\sum_{i,j} \int_{B_1(x_0)} |\partial_{ij} \phi_{R_n}^c|^2 \, dx \leq C \bigg( \| \nabla \phi_{R_n}^c \|_{L^2_\mathrm{uloc}(\mathbb R^3)}^2 + \| m_{R_n}^c \|_{L^2_\mathrm{uloc}(\mathbb R^3)}^2 \\
&\qquad \qquad + \| u_{R_n} \|_{H^1_\mathrm{uloc}(\mathbb R^3)}^4 + \sum_{y \in \mathbb P \cap B_{3}(x_0)} c_y^2 \bigg) + \frac12 \sum_{i,j} \| \partial_{ij} \phi_{R_n}^c \|_{L^2_\mathrm{uloc}(\mathbb R^3)}^2.
\end{align*}
Using \eqref{eqh1boundurn} and \eqref{eqgradphirnc}, we now get for all $1 \leq i,j \leq 3$ the bound
\begin{equation}
\label{eqd2phirnc}
\| \partial_{ij} \phi_{R_n}^c \|_{L^2_\mathrm{uloc}(\mathbb R^3)} \leq C \big( 1 + M^\frac52 \big).
\end{equation}

By an elementary calculation with $p \in [1, \frac32)$, one easily obtains 
\begin{align*}
&\int_{B_1(x_0)} \bigg| \nabla \sum_{y \in \mathbb P \cap B_{R_n}(0)} \frac{c_y}{|\cdot - \, y|} \omega(\cdot - y) \bigg|^p dx \\
&\qquad\leq C_p \sum_{y \in \mathbb P \cap B_{3}(x_0)} c_y^p \int_{B_1(x_0)} \Big( \frac{1}{|\cdot - \, y|^{2p}} + \frac{1}{|\cdot - \, y|^{p}} \Big) dx \leq C_p \sum_{y \in \mathbb P \cap B_{3}(x_0)} c_y^p
\end{align*}
with a constant $C_p > 0$. This gives rise to 
\[
\bigg\| \nabla \sum_{y \in \mathbb P \cap B_{R_n}(0)} \frac{c_y}{|\cdot - \, y|} \omega(\cdot - y) \bigg\|_{L^p_\mathrm{uloc}(\mathbb R^3)} \leq C_p \, M
\]
for $1<p<\frac{3}{2}$
and --- by taking into account \eqref{eqgradphirnc} ---
\[
\| \nabla \phi_{R_n} \|_{L^p_\mathrm{uloc}(\mathbb R^3)} \leq C_p \big( 1 + M^\frac74 \big).
\]
Similarly, we deduce
\begin{align*}
&\int_{B_1(x_0)} \eta_\rho^2 |\partial_i \phi_{R_n}|^2 \, dx \\
&\quad\leq C \sum_{y \in \mathbb P \cap B_{3}(0)} c_y^2 \int_{B_1(x_0)} \eta_\rho^2 \Big( \frac{1}{|\cdot - \, y|^{4}} + \frac{1}{|\cdot - \, y|^{2}} \Big) dx + C \int_{B_1(x_0)} \eta_\rho^2 |\partial_i \phi_{R_n}^c|^2 \, dx \\
&\quad\leq C_\rho M^2 + C \big(1 + M^\frac72 \big) \leq C_\rho \big(1 + M^\frac72 \big)
\end{align*}
by employing \eqref{eqgradphirnc}. This yields
\[
\| \eta_\rho \partial_i \phi_{R_n} \|_{L^2_\mathrm{uloc}(\mathbb R^3)} \leq C_\rho \big(1 + M^\frac74 \big).
\]
An analogous argumentation using \eqref{eqd2phirnc} finally leads to
\[
\| \eta_\rho \partial_{ij} \phi_{R_n} \|_{L^2_\mathrm{uloc}(\mathbb R^3)} \leq C_\rho \big(1 + M^\frac52 \big).
\]
This finishes the proof.
\end{proof}

\begin{proposition}
\label{propbounds}
Let $m = m_c + \sum_{x \in \mathbb P} c_x \delta_x$ be a charge distribution subject to assumption (A1).
Then, there exists a solution $(u, \phi) \in H^1_\mathrm{uloc}(\mathbb R^3) \times L^2_\mathrm{uloc}(\mathbb R^3)$, $u \geq 0$, to
\begin{equation}
\label{eqtf}
\begin{cases}
-\Delta u + \frac53 u^\frac73 - \phi u = 0, \\
-\Delta \phi = 4\pi (m - u^2),
\end{cases}
\end{equation}
in the distributional sense. Furthermore, using the cutoff $\eta_\rho$ introduced in Notation~\ref{noteta} this solution satisfies the bounds
\begin{equation*}
\begin{split}
\| u \|_{H^2_\mathrm{uloc}(\mathbb R^3)} &\leq C (1 + M^4), \\
\| u \|_{L^p_\mathrm{uloc}(\mathbb R^3)} &\leq C_p (1 + M^\frac34) \mbox{\ for all\ } 1 \leq p < 4, \\
\| \phi \|_{L^p_\mathrm{uloc}(\mathbb R^3)} &\leq C_p (1 + M) \mbox{\ for all\ } 1 \leq p < 3, \\
\| \phi \|_{W^{1,p}_\mathrm{uloc}(\mathbb R^3)} &\leq C_p \big(1 + M^\frac74 \big) \mbox{\ for all\ } 1 \leq p < \frac32, \\
\| \eta_\rho \partial_i \phi \|_{L^2_\mathrm{uloc}(\mathbb R^3)} &\leq C_\rho \big( 1 + M^\frac74 \big) \mbox{\ for all\ } 0 < \rho < \ol \rho, \ 1 \leq i \leq 3, \\
\| \eta_\rho \partial_{ij} \phi \|_{L^2_\mathrm{uloc}(\mathbb R^3)} &\leq C_\rho \big( 1 + M^\frac52 \big) \mbox{\ for all\ } 0 < \rho < \ol \rho, \ 1 \leq i,j \leq 3
\end{split}
\end{equation*}
where $C, \, C_p, \, C_\rho >0$ are independent of $M$.
\end{proposition}

\begin{proof}
This proposition can be proven along the same lines of arguments as a similar statement in \cite{NO17}. We first set $R_n \coleq R_0 + n$ for $n \in \mathbb N$ in Proposition \ref{proprn} and obtain bounded sequences $u_{R_n} \in H^2_\mathrm{uloc}(\mathbb R^3)$ and $\phi_{R_n} \in W^{1,p}_\mathrm{uloc}(\mathbb R^3)$, $p \in (1,\frac32)$.
By a diagonal sequence argument, we get subsequences $u_{R_n} \geq 0$ weakly converging in $H^2(B_R(0))$ to some $u \in H^2_\mathrm{loc}(\mathbb R^3)$ and $\phi_{R_n}$ weakly converging in $W^{1,p}(B_R(0))$ to some $\phi \in W^{1,p}_\mathrm{loc}(\mathbb R^3)$ for all $R>0$ and $p \in (1, \frac32)$.

Let $x_0 \in \mathbb R^3$. We then have $u_{R_n} \rightharpoonup u$ in $H^2(B_1(x_0))$ and $\phi_{R_n} \rightharpoonup \phi$ in $W^{1,p}(B_1(x_0))$ for all $p \in [1,\frac32)$; in particular, one derives $u_{R_n} \rightarrow u$ in $L^q(B_1(x_0))$ for all $q \in[1, 4)$ and $\phi_{R_n} \rightarrow \phi$ in $L^{r}(B_1(x_0))$ for all $r \in [1,3)$. The corresponding bounds on $u$ and $\phi$ are now an immediate consequence of the bounds on $u_{R_n}$ and $\phi_{R_n}$ in \eqref{eqbounds}.

The $H^2$-type bound on $\phi$  on the set $\mathbb R^3 \backslash \mathbb P$ can be deduced by a similar reasoning. We start
by observing that
\[
\| \partial_i \phi_{R_n} \|_{L^2(B_R(0) \cap \mathrm{int} \{ \eta_\rho = 1\})}
\leq \| \eta_\rho \partial_i \phi_{R_n} \|_{L^2(B_R(0))} \leq C_\rho R^\frac32 \big(1 + M^\frac74 \big),
\]
for all $1 \leq i,j \leq 3$, $R > 0$, and $0 < \rho < \ol \rho$ due to the bounds in \eqref{eqbounds}. 
By selecting a diagonal sequence $\phi_{R_n}$, we find that $\phi_{R_n}$ weakly converges to $\phi$ in $H^1(B_R(0) \cap \mathrm{int} \{ \eta_{\rho/2} = 1\})$ for all $R > 0$ and $0 < \rho < \ol \rho$. This fact gives rise to
\[
\| \eta_\rho \partial_i \phi \|_{L^2(B_1(x_0))} \leq \| \partial_i \phi \|_{L^2(B_1(x_0) \cap \mathrm{int} \{ \eta_\frac{\rho}{2} = 1\})} \leq C_\frac{\rho}{2} \big(1 + M^\frac74 \big)
\]
(and an analogous bound on $\eta_\rho \partial_{ij} \phi$) for all $x_0 \in \mathbb R^3$, $0 < \rho < \ol \rho$, and $1 \leq i,j \leq 3$.

We subsequently rewrite \eqref{eqtfrn}
in the distributional formulation. For all $v \in C_c^\infty(\mathbb R^3)$, we have
\[
\int_{\mathbb{R}^3} \Big( -u_{R_n} \Delta v + \frac53 u_{R_n}^\frac73 v - \phi_{R_n} u_{R_n} v \Big) \, dx = 0
\]
and
\[
-\int_{\mathbb R^3} \phi_{R_n} \Delta v \, dx = 4\pi \bigg( \sum_{x \in \mathbb P \cap B_{R_n}(0)} c_x v(x) + \int_{\mathbb R^3} \big(m_{c,R_n} - u_{R_n}^2 \big) v \, dx \bigg).
\]
Due to the convergence properties of $u_{R_n}$ and $\phi_{R_n}$ derived above, these equations converge to the corresponding distributional formulation of \eqref{eqtf} for $n \rightarrow \infty$.
\end{proof}

Note that in the literature sometimes a condition equivalent to (A1) is used.
\begin{remark}
\label{reminfboundm}
We now give an equivalent characterization of the $\inf$-condition for charge distributions $m$ in (A1), which also appears in \cite{CLBL98}. An analogous statement without Dirac measures has been proven in \cite{NO17}. But as the result only appeals to the mass, the proof is the same.

Let $m = m_c + \sum_{y \in \mathbb P} c_y \delta_y$ where $m_c \in L^2_\mathrm{uloc}(\mathbb R^3)$, $m_c \geq 0$, $c_{y} > 0$ and $\mathbb P \subset \mathbb R^3$ such that $|x-y| \geq 4\rho$ for all $x, y \in \mathbb P$, $x \neq y$, for some $\rho>0$. Then, the following statements are equivalent.
\begin{enumerate}
	\item[(i)] $\begin{aligned} \inf_{x \in \mathbb R^3} \Big( \int_{B_R(x)} m_c \, dy + \sum_{y \in \mathbb P \cap B_R(x)} c_y \Big) \geq \omega_0 R^3 \ \mbox{for all} \ R \geq \omega_0^{-1} \end{aligned}$ \\[1ex]
	\item[(ii)] $\begin{aligned} \lim_{R \rightarrow \infty} \inf_{x \in \mathbb R^3} \frac{1}{R} \Big( \int_{B_R(x)} m_c \, dy + \sum_{y \in \mathbb P \cap B_R(x)} c_y \Big) = \infty \end{aligned}$
\end{enumerate}
\end{remark}

\begin{theorem}
\label{theoreminf}
Let the charge distribution $m$ satisfy (A1). Then, there exists a unique solution $(u, \phi) \in H^1_\mathrm{uloc}(\mathbb R^3) \times L^2_\mathrm{uloc}(\mathbb R^3)$, $u \geq 0$, to
\begin{equation*}
\begin{cases}
-\Delta u + \frac53 u^\frac73 - \phi u = 0, \\
-\Delta \phi = 4\pi (m - u^2),
\end{cases}
\end{equation*}
in the distributional sense. This solution $(u,\phi)$ satisfies the bounds established in Proposition \ref{propbounds} as well as 
\[
\inf_{x \in \mathbb R^3} u(x) \geq c
\]
where $c>0$ only depends on $\rho$, $M$, and $\omega_0$.
\end{theorem}

\begin{proof}
The existence of a corresponding solution has already been proven in Proposition \ref{propbounds}, whereas the uniqueness follows from the general existence and uniqueness result in \cite[Theorem 6.10]{CLBL98}. The assumptions in this theorem are satisfied due to (A1) and Remark \ref{reminfboundm}.

As in \cite{NO17}, we assume that 
\[
\inf_{m \text{ subject to } (A1)} ~~\inf_{x \in \mathbb R^3} u(x) = 0
\]
and show that this assumption leads to a contradiction. We choose a sequence of charges $m_n$ satisfying (A1) and $x_n \in \mathbb R^3$ such that the solution $(u_n, \phi_n)$ fulfills
\[
u_n(x_n) \leq \frac1n.
\]
Using the bounds on $u_n$ and $\phi_n$ from Proposition \ref{propbounds}, we estimate
\[
\Big\| \frac53 u_n^\frac43 - \phi_n \Big\|_{L^2_\mathrm{uloc}(\mathbb R^3)} \leq \frac53 \| u_n \|_{L^\frac83_\mathrm{uloc}(\mathbb R^3)}^\frac43 + \| \phi_n \|_{L^2_\mathrm{uloc}(\mathbb R^3)} \leq C (1 + M).
\]
From Harnack's inequality \cite[Corollary 5.2]{Tru73} and the uniform bound on the coefficient of the operator $-\Delta + \frac53 u_n^\frac43 - \phi_n$, we obtain a constant $C > 0$ depending only on $M$ and $R$ such that 
\begin{equation}
\label{eqharnack}
\sup_{x \in B_R(x_n)} u_n(x) \leq C \inf_{x \in B_R(x_n)} u_n(x) \leq \frac{C}{n}
\end{equation}
for all $R>0$. The shifted functions $u_n( \cdot + x_n)$, thus, converge uniformly to zero on $B_R(0)$, while the potential $\phi_n$ solves 
\begin{equation}
\label{eqpotn}
-\Delta \phi_n = 4\pi (m_n - u_n^2)
\end{equation}
in the sense of distributions.

We now choose a cut-off function $\omega \in C_c^\infty(\mathbb R^3)$ subject to $0 \leq \omega \leq 1$, $\omega = 1$ on $B_\frac{1}{2}(0)$, and $\omega = 0$ on $\mathbb R^3 \backslash B_1(0)$. By testing \eqref{eqpotn} with $\omega(\frac{\cdot - x_n}{R})$, we derive
\begin{align*}
&4 \pi \int_{B_R(x_n)} m_n(x) \, \omega\Big(\frac{x - x_n}{R}\Big) \, dx \\
&\qquad \qquad = 4 \pi \int_{B_R(x_n)} u_n^2 \omega\Big(\frac{x - x_n}{R}\Big) \, dx - \frac{1}{R^2} \int_{B_R(x_n)} \phi_n \Delta \omega\Big(\frac{x - x_n}{R}\Big) \, dx.
\end{align*}
As a consequence of (A1) and the bound on $\phi_n$ from Proposition~\ref{propbounds}, we may now estimate
\[
c R^3 \leq \int_{B_\frac{R}{2}(x_n)} m_n(x) \, dx \leq \int_{B_R(x_n)} u_n^2(x) \, dx + CR (1 + M)
\]
with positive constants $c$ and $C$ independent of $M$ and $R \geq 2\omega_0^{-1}$. However, if we first choose $R\geq 2\omega_0^{-1}$ solving $cR^3 \geq 1 + CR(1+M)$, and then $n \in \mathbb N$ such that $\int_{B_R(x_n)} u_n^2(x) \, dx < 1$ holds true (according to \eqref{eqharnack}), we arrive at a contradiction.
\end{proof}

\bibliographystyle{abbrv}
\bibliography{TFWModel}

\end{document}